\documentclass[article,ij4uq]{ij4uq}    

\usepackage[hang]{footmisc}
\fancypagestyle{plain}{%
  }

\renewcommand{\dmyy}{17}
\renewcommand{\myyear}{2019}
\renewcommand{\today}{}

\def\L{{ \mathcal{L}}}

\def\D{{ \mathcal{D}}}

\def\K{{ \mathcal{K}}}

\def\P{{ \mathcal{P}}}

\def\B{{ \mathcal{B}}}

\def\F{{ \mathcal{F}}}

\def\R{{\mathbb{R}}}

\newtheorem{example}{Example}[section]

\begin{document}

\volume{}
\title{Hyper-Differential Sensitivity Analysis of Uncertain Parameters in PDE-Constrained Optimization}
\titlehead{HDSA for PDE-Constrained Optimization}
\authorhead{Joseph Hart, Bart van Bloemen Waanders, and Roland Herzog}

\author[1]{Joseph Hart}
\corrauthor[1]{Bart van Bloemen Waanders}
\corremail{bartv@sandia.gov}

\address[1]{Optimization and Uncertainty Quantification, Sandia National Laboratories, P.O. Box 5800, Albuquerque, NM 87123-1320}
\author[2]{Roland Herzog}
\address[2]{Technical University Chemnitz, Faculty of Mathematics, 09107 Chemnitz, Germany}

\abstract{

Many problems in engineering and sciences require the solution of
large scale optimization constrained by partial differential equations
(PDEs). Though PDE-constrained optimization is itself challenging,
most applications pose additional complexity, namely, uncertain
parameters in the PDEs. Uncertainty quantification (UQ) is necessary
to characterize, prioritize, and study the influence of these
uncertain parameters.  Sensitivity analysis, a classical tool in UQ,
is frequently used to study the sensitivity of a model to uncertain
parameters. In this article, we introduce ``hyper-differential
sensitivity analysis" which considers the sensitivity of the solution
of a PDE-constrained optimization problem to uncertain parameters.
Our approach is a goal-oriented analysis which may be viewed as a tool
to complement other UQ methods in the service of decision making and
robust design. We formally define hyper-differential sensitivity
indices and highlight their relationship to the existing optimization
and sensitivity analysis literatures.  Assuming the presence of low
rank structure in the parameter space, computational efficiency is
achieved by leveraging a generalized singular value decomposition in
conjunction with a randomized solver which converts the computational
bottleneck of the algorithm into an embarrassingly parallel loop.  Two
multi-physics examples, consisting of nonlinear steady state control
and transient linear inversion, demonstrate efficient identification
of the uncertain parameters which have the greatest influence on the
optimal solution.  }


\maketitle

\section{Introduction}

Many critical applications in science and engineering require the
analysis of multi-physics phenomena across several spatial and
temporal scales. For instance in material science, fusion energy,
hydrocarbon extraction, and climate science, the ultimate goal is to
solve large scale optimization problems while reconciling
uncertainties that arise in constituent models, material properties,
boundary conditions, initial conditions, and multi-phyiscs
interfaces.Although the desire is to apply modeling, uncertainty
quantification, and optimization to arrive at robust solutions, the
culmination of such analysis poses a formidable computational
challenge rendering many algorithmic strategies ineffective.

In this article, we propose a sensitivity analysis framework to
determine the importance of uncertain parameters in the context of
optimal solutions. To aid in distinguishing our approach from existing
sensitivity analysis techniques, we introduce the term
``hyper-differential sensitivity analysis" or HDSA for short. We
compute the Fr\'echet derivative of the solution of a PDE-constrained
optimization problem with respect to uncertain parameters. Assuming a
low rank structure in the parameter space, a truncated generalized
singular value decomposition is employed to efficiently estimate
sensitivity indices which enables an identification of inconsequential
uncertain parameters and a prioritization of the uncertainties. By
identifying low dimensional structure in high dimensional parameter
spaces, HDSA is a goal-oriented analysis which considers the
sensitivity of the optimal solution.

The complexities of PDE-constrained optimization consists of
incorporating the discretization of PDEs as well as implementing
complicated linear algebra constructs, including adjoints and
Hessians.  Furthermore, the large scale nature of these problems
requires efficient computational implementation with scalable parallel
linear algebra.  Considerable research and development has been
conducted and the interested reader is referred to a small sampling of
the literature \cite{Vogel_99,
  Archer_01,Haber_01,Vogel_02,Biegler_03,Biros_05,Laird_05,Hintermuller_05,Hazra_06,Biegler_07,Borzi_07,Hinze_09,Biegler_11}. Although
a range of algorithmic strategies are possible, we use standard
methods to solve the underlying PDE-constrained optimization problem,
consisting of Newton-based solvers, trust region globalization, Tikhonov
regularization, finite element discretization of the PDE constraints,
and matrix-free operators. The focus of this paper is sensitivity
analysis with respect to the solution of PDE-constrained problems and
therefore inherits all the associated complexities.

Traditional sensitivity analysis can be divided into two subfields:
local sensitivity analysis and global sensitivity analysis
\cite{intro_SA_uq_handbook,saltellibook}. There are a plurality of
methods within each subfield; we highlight one class of methods from
each to provide background for this article. Local sensitivity
analysis studies the influence of uncertain parameters on a quantity
of interest (for instance, a functional of the PDE solution) at some
fixed parameter value. Computing the derivatives of the quantity of
interest with respect to the parameters is one measure of local
sensitivity. A large derivative indicates that the quantity of
interest is sensitive to the parameter. In the context of
PDE-constrained optimization, finite difference, direct, and
adjoint-based sensitivities are local sensitivities of the objective
function with respect to the optimization variables, whereas HDSA is
the local sensitivity of the optimal solution with respect to
uncertain parameters. The limitation of local sensitivity analysis is
that it is only valid in a neighborhood of the fixed parameter
value. Global sensitivity analysis
\cite{iooss,borgonovo2,kucherenko_derivative,variance_based_uq_handbook,staum}
seeks to alleviate this problem by varying the parameters over a set
(typically with an associated probability measure) and measuring the
importance of the parameters by averaging over this set. The approach
of \cite{kucherenko_derivative} connects local sensitivity analysis to
global sensitivity analysis. Since the derivative is local in the
sense that it depends on the user specifying a nominal parameter
value, a global approach computes the expected value (in parameter
space) of the squared derivative. Both local and global
derivative-based (hyper-differential) sensitivity indices are defined
in this article, though additional complexities arise in the context
of HDSA which do not occur in the traditional sensitivity analysis
framework.

Building on the work of Brandes and Griesse\footnote{R. Herzog is
  formerly R. Griesse} \cite{griesse2} (and related work
\cite{griesse_constraints,Griesse_part_2,Griesse_part_1,Griesse_Thesis,Griesse_SISC,griesse_3d,Griesse_AD}),
we define sensitivity indices, provide algorithmic developments, and
implement software which enables analysis for large-scale optimization
problems with high dimensional uncertain parameter spaces. The
evaluation of these sensitivities requires an eigenvalue computation
involving the optimality system.  To that end, we have reformulated
the eigenvalue problem and introduced the use of a randomized
eigenvalue solver which allows for parallelization of the underlying
matrix-vector products \cite{arvind}.  Our implementation is in C++
with parallel linear algebra constructs.  The parallel randomized
solver extends the use of the native parallelism for a second level of
parallelism. Having developed an efficient computational framework, we
introduce global sensitivity indices in the context of PDE-constrained
optimization, demonstrate how they may be estimated, and highlight the
challenges involved. Finally, we demonstrate our approach on the
control of a thermal-fluid flow multi-physics problem and a source
inversion problem constrained by Darcy flow and advection diffusion.
The main contributions of this paper are: 1) the introduction of local
and global hyper-differential sensitivity indices as tools to analyze
the sensitivity of optimal solutions in the service of robust design
and decision making, 2) the development of C++ software infrastructure
to leverage state of the art (matrix free) PDE-constrained
optimization and parallel linear algebra, and 3) the formulation of a
symmetric generalized eigenvalue problem (to estimate sensitivities)
and its numerical solution via randomized methods to achieve (nearly)
embarrassingly parallel efficiency.

The article is organized as follows. We first define
hyper-differential local and global sensitivities in
Section~\ref{sensitivity_analysis}. Our algorithmic contributions to
accelerate the computation of local sensitivities are given in
Section~\ref{sec:comp_local_sen}, followed by
Section~\ref{sec:algo_overview} which overviews our algorithms,
provides an analysis of computational cost, and a strategy for
interpreting the sensitivities. Our proposed framework is demonstrated
through two applications in Section~\ref{sec:numerical_results}. We
provide conclusions and highlight areas of future work in
Section~\ref{sec:conclusion}.

\section{Hyper-Differential Sensitivity Analysis for Solutions of PDE-Constrained Optimization Problems}
\label{sensitivity_analysis}
This section provides necessary background and formally defines
hyper-differential local and global sensitivities. The implicit
function theorem forms the mathematical foundation for the proposed
approach. The concept of sensitivities, defined through the implicit
function theorem, are present in different portions of the
optimization literature with multiple titles and in various
contexts. For instance, it is called post optimality analysis,
sensitivity analysis, or parametric programming in the operations
research literature \cite{Murthy_OR}, sensitivity analysis, stability
analysis, and perturbation analysis in the nonlinear programming
community \cite{shapiro_SIAM_review}, and parametric sensitivity
analysis in PDE-constrained optimal control
\cite{Griesse_part_2,Griesse_part_1}.  We follow Brandes and
Griesse~\cite{griesse2} who introduced such sensitivity analysis for
PDE-constrained optimization.

\subsection{Local Sensitivity Analysis}
\label{sec:loc_sen}
 Consider the PDE-constrained optimization problem
\begin{align} 
\label{opt_gen}
& \min\limits_{u,z} J(u,z,\theta) \\
\text{s.t.} \ & c(u,z,\theta) = 0 \nonumber \\
& u \in U, z \in Z \nonumber
\end{align}
where $U$ is the state space, $Z$ is the optimization variable space,
$\theta \in \Theta$ are uncertain parameters which are fixed in the
optimization problem, $J:U \times Z \times \Theta \to \R$ is an
objective function, $c:U \times Z \times \Theta \to \Lambda^\star$ is
the weak form of a PDE, and $\Lambda^\star$ denotes the dual of
$\Lambda$. It is assumed that $U$ and $\Lambda$ are reflexive Banach
spaces, and $Z$ and $\Theta$ are Hilbert spaces. The stronger
assumptions on $Z$ and $\Theta$ are to enable subsequent analysis with
the generalized singular value decomposition (GSVD). The spaces
$U,Z,\Theta$, and $\Lambda$ may be finite or infinite dimensional. In
particular, the optimization variable space $Z$ may correspond to
euclidean space or a function space in either a control, design, or
inverse problem setting, and the parameter space $\Theta$ may
represent a finite number of parameters or a functional representation
(for instance, a spatially dependent parameter), in which case
$\Theta$ is infinite dimensional. The PDE represented by $c$ may be
time dependent or in steady state. We assume that $J$ and $c$ are
twice continuously differentiable with respect to $(u,z,\theta)$ and
that the Fr\'echet derivative of $c$ with respect to $u$ is
surjective.

Our goal is to develop a computational framework which may by used for
large scale PDE systems with infinite (or large finite) dimensional
parameter uncertainty. In particular, appealing to the use of HDSA for
transient, nonlinear, and multi-physics systems with spatially and/or
temporally dependent parameters where it may identify important
structure embedded in complex physics.

To analyze \eqref{opt_gen} as an unconstrained problem, define the
Lagrangian $\L:U \times Z \times \Lambda \times \Theta \to \R$ as
\begin{eqnarray}
\label{lagrangian}
\L(u,z,\lambda,\theta) = J(u,z,\theta) + \langle \lambda, c(u,z,\theta) \rangle ,
\end{eqnarray}
where $\lambda \in \Lambda$ is the (unique) Lagrange multiplier, in the context of
PDE-constrained optimization it is called the adjoint state. Throughout the article
the subscripts $J_*$ and $c_*$ are used to denote the Fr\'echet derivative of
$J$ and $c$ with respect to $*$, respectively. After discretizing,
$J_*$ and $c_*$ will denote the derivative of $J$ and Jacobian of $c$,
respectively.

We seek to perform derivative-based analysis, specifically the
derivative of the optimal solution with respect to the
parameters. Care must be taken since \eqref{opt_gen} may admit
multiple local minima. The following result (Lemma 2.6 in
\cite{griesse2}) is foundational for our sensitivity analysis.

Let $(u_0,z_0)$ be a local minimum of \eqref{opt_gen} when
$\theta=\theta_0$ and $\lambda_0$ be the unique adjoint
state. Assuming the second order sufficient optimality condition
holds, see \cite{griesse2} for details, there exist neighborhoods
$\mathcal N(\theta_0) \subset \Theta$ and $\mathcal
N(u_0,z_0,\lambda_0) \subset U \times Z \times \Lambda$ and a
continuously differentiable function
\begin{eqnarray*}
\F:\mathcal N(\theta_0) \to \mathcal N(u_0,z_0,\lambda_0)
\end{eqnarray*}
such that for all $\theta \in \mathcal N(\theta_0)$, $\F(\theta)=(u_{opt}(\theta),z_{opt}(\theta),\lambda_{opt}(\theta))$ is the unique stationary point of $\L(\cdot,\cdot,\cdot,\theta)$ in $\mathcal N(u_0,z_0,\lambda_0)$, that is,
\begin{eqnarray*}
\L_{(u,z,\lambda)}(\F(\theta),\theta)=0.
\end{eqnarray*}
The Fr\'echet derivative of $\F$ at $\theta_0$ is given by
\begin{eqnarray*}
\F'(\theta_0;z_0) = \K(\theta_0;z_0)^{-1} \B(\theta_0;z_0)
\end{eqnarray*}
where
\begin{eqnarray}
\label{K_B_ops} 
\K(\theta_0;z_0) = \left( \begin{array}{ccc}
\L_{u,u} & \L_{u,z} & c_u^\star \\
\L_{z,u} & \L_{z,z} & c_z^\star \\
c_u & c_z & 0 \\
\end{array} \right)
\qquad \text{and} \qquad
\B(\theta_0;z_0) = - \left( \begin{array}{c}
\L_{u,\theta} \\
\L_{z,\theta} \\
\L_{\lambda,\theta} \\
\end{array} \right),
\end{eqnarray}
with the second derivatives of $\L$ evaluated at
$(u_0,z_0,\lambda_0,\theta_0)$; $\star$ denotes the adjoint of an
operator. We only write explicit dependence on $(\theta_0;z_0)$ to
simplify notation. The operator $\K$ is the Karush-Kuhn-Tucker (KKT)
operator.

This gives a computable expression for the change in the optimal
solution when the parameters $\theta_0 \in \Theta$ are perturbed. We
emphasize that this is a local result in the sense that the
optimization problem \eqref{opt_gen} may have many stationary points;
$\F$ is only defined in the neighborhood of a particular stationary
point, $\mathcal N(u_0,z_0,\lambda_0)$, ensuring existence of the
mapping and its derivative.

The directional derivative of optimal solution in the direction
$\theta$ is given by the solution, $(u,z,\lambda)$, of the linear
system
\begin{eqnarray}
\label{KKT_solve}
\K(\theta_0;z_0)
\left( \begin{array}{c}
u\\
z \\
\lambda \\
\end{array} \right)
= \B(\theta_0;z_0) \theta .
\end{eqnarray}

To determine the sensitivity of the optimization variables to changes
in $\theta$, define the projection operator $\P: U \times Z \times
\Lambda \to Z$ by
\begin{eqnarray*}
\P 
 \left( \begin{array}{c}
u \\
z\\
\lambda \\
\end{array} \right)
=
z .
\end{eqnarray*}
Then the Fr\'echet derivative of the optimal solution, $z$, with
respect to the parameters $\theta$ may be expressed as the linear
operator $\D(\theta_0;z_0):\Theta \to Z$,
\begin{eqnarray}
\label{sen_operator}
\D(\theta_0;z_0) = \P \K(\theta_0;z_0)^{-1} \B(\theta_0;z_0) .
\end{eqnarray}
Note that the sensitivity of the state, or a function of the state,
adjoint, or combination of them may be considered by using a different
$\P$. To simplify the presentation, this article focuses on the
sensitivity with respect to the optimization variables $z$.

We introduce the hyper-differential local sensitivity function
$\mathcal S(\theta_0;z_0):\Theta \to \R$ which is defined by
\begin{eqnarray}
\label{sensitivity_function}
\mathcal S(\theta_0;z_0) \phi = \left\vert \left\vert \D(\theta_0;z_0) \frac{\phi}{\vert \vert \phi \vert \vert_\Theta} \right\vert \right\vert_Z \qquad \forall \phi \in \Theta.
\end{eqnarray}
The scalar $\mathcal S(\theta_0;z_0) \phi$ may be interpreted as the
magnitude of the change in the optimal $z$ when the parameters are
perturbed in the direction $\phi$. Throughout the article we will
commonly refer to \eqref{sensitivity_function} as a local sensitivity
for short, but adopt the formal title ``hyper-differential local
sensitivity" to distinguish this approach from other forms of
derivative-based sensitivity analysis.

A traditional sensitivity analysis approach may solve the optimization problem with $\theta=\theta_0$ to determine $z_0$, fix $z=z_0$, 
and analyze the sensitivity of the state (or
objective function) to changes in the parameters; there are no 
KKT solves involved. This suffers two drawbacks:
\begin{enumerate}
\item[$\bullet$] The optimal $z$ will change when the parameters change, so fixing the optimization variable may lead to sensitivity which occurs because of a poor $z$.
\item[$\bullet$] There may be parameters for which the state $u$ is sensitive but the optimal $z$ is not, or vice versa. Since the solution of the optimization problem $z$ is our end goal, fixing the $z$ does not give the needed sensitivities. 
\end{enumerate}

Example~\ref{two_sensitivity_approaches} illustrates the difference
between this traditional sensitivity analysis approach versus
hyper-differential sensitivity analysis. We argue that the latter is
applicable to many engineering and science problems, and while it is
computationally intensive, the final sensitivity results are in the
context of control, design, and inversion goals.

\begin{example}
\label{two_sensitivity_approaches}
\textit{Consider the optimization problem}
\begin{align}
\label{example_opt}
& \min_{u,z} \hspace{3 mm} J(u,z) = (u-2)^2+.0005z^2 \\
& s.t. \hspace{5 mm} u=\frac{1}{1+e^{-\theta_1z}}+\theta_2 \nonumber
\end{align}
\textit{where $\theta=(\theta_1,\theta_2)$ are uncertain parameters.}
 
\textit{We solve \eqref{example_opt} with $\theta=(0.5,0.5)$ to find the optimal solution $z=8.22$. Following the framework described above, define the optimal solution as a function of $\theta$, defined on a neighborhood of $(0.5,0.5)$, as $z_{opt}(\theta) = \P \F(\theta)$. Then}
 
\begin{eqnarray*}
\left\vert \frac{\partial z_{opt} }{\partial \theta_1}(0.5,0.5) \right\vert = 9.99 \qquad \text{and} \qquad \left\vert \frac{\partial z_{opt} }{\partial \theta_2}(0.5,0.5) \right\vert = 3.12.
\end{eqnarray*}

\textit{Alternatively, we may consider the reduced objective function parameterized by $\theta$ and evaluated at $z=8.22$,}
\begin{eqnarray*}
g(\theta_1,\theta_2) = \left(\frac{1}{1+e^{-\theta_1(8.22)}}+\theta_2-2\right)^2 + .0005 (8.22)^2.
\end{eqnarray*}
\textit{Computing the partial derivatives of $g$ with respect to $\theta_1$ and $\theta_2$ gives}
\begin{eqnarray*}
\left\vert \frac{\partial g }{\partial \theta_1}(0.5,0.5) \right\vert = 0.135 \qquad \text{and} \qquad \left\vert \frac{\partial g }{\partial \theta_2}(0.5,0.5) \right\vert = 1.03 .
\end{eqnarray*}
 
\textit{Hence computing the sensitivity of the optimal solution $z_{opt}$ with respect to $\theta$ gives
  a different conclusion than the sensitivity of the objective
  function, evaluated at the optimal solution, with respect to $\theta$. This conclusion highlights the difference between hyper-differential sensitivity analysis and traditional approaches to sensitivity analysis.}
   
\textit{Figure~\ref{fig:logistic_example} gives some intuition for
  this result. The constraint is plotted with $u$ as a function of
  $z$. Curves are plotted for different values of $\theta$ and the
  solution of \eqref{example_opt} for each fixed $\theta$ is given by
  the dot on the curve. In the left (right) panel $\theta_2=0.5$
  ($\theta_1=0.5$) is fixed and $\theta_1$ ($\theta_2$) varies from
  $0.3$ to $0.7$. This demonstrates that as $\theta_1$ varies, left
  panel, the optimal solution $z$ varies significantly while the state
  $u$ is kept nearly constant; whereas as $\theta_2$ varies, right
  panel, $z$ is nearly constant while $u$ varies significantly.}

\begin{figure}
\centering
  \includegraphics[width=0.49\textwidth]{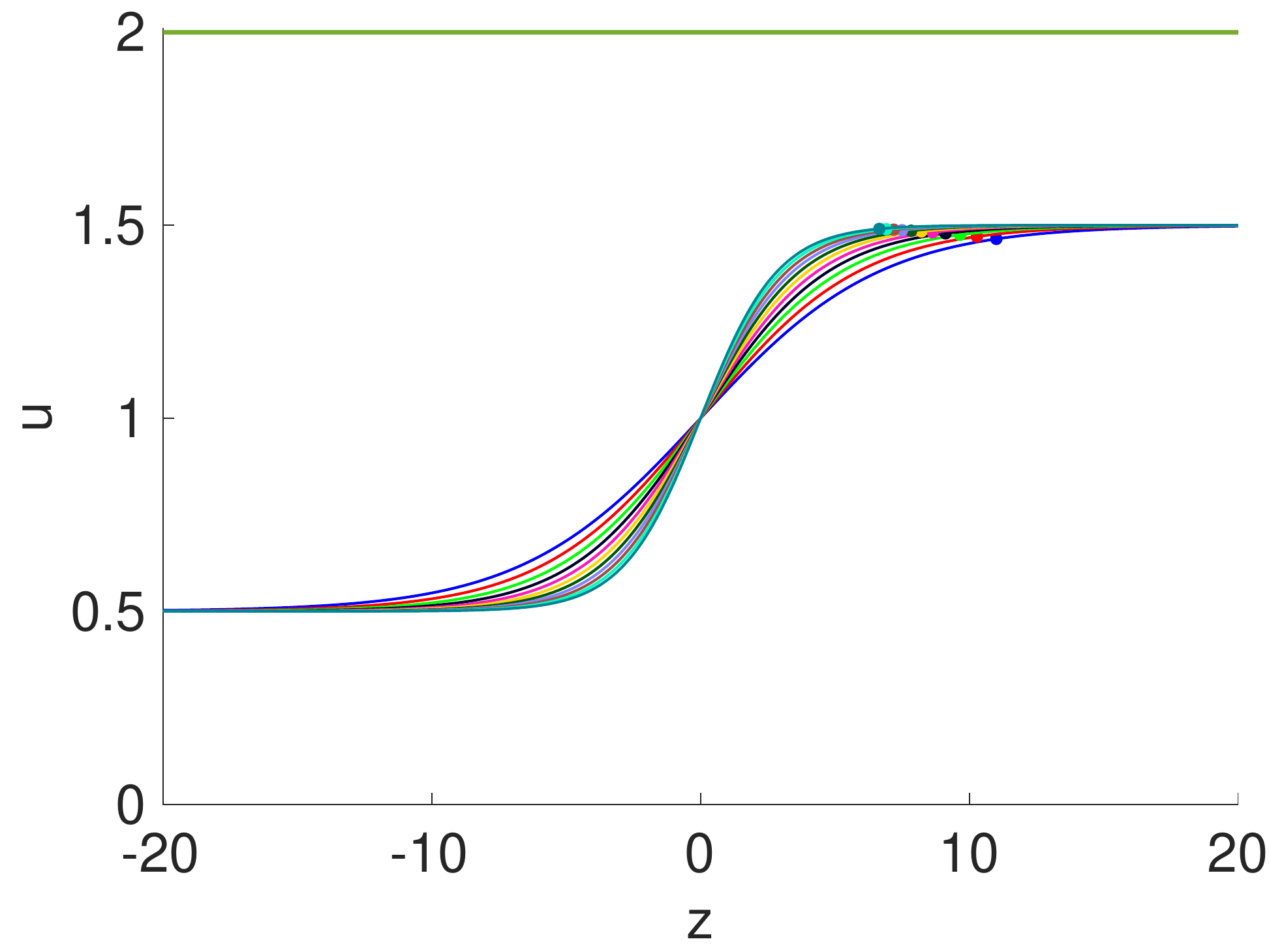}
    \includegraphics[width=0.49\textwidth]{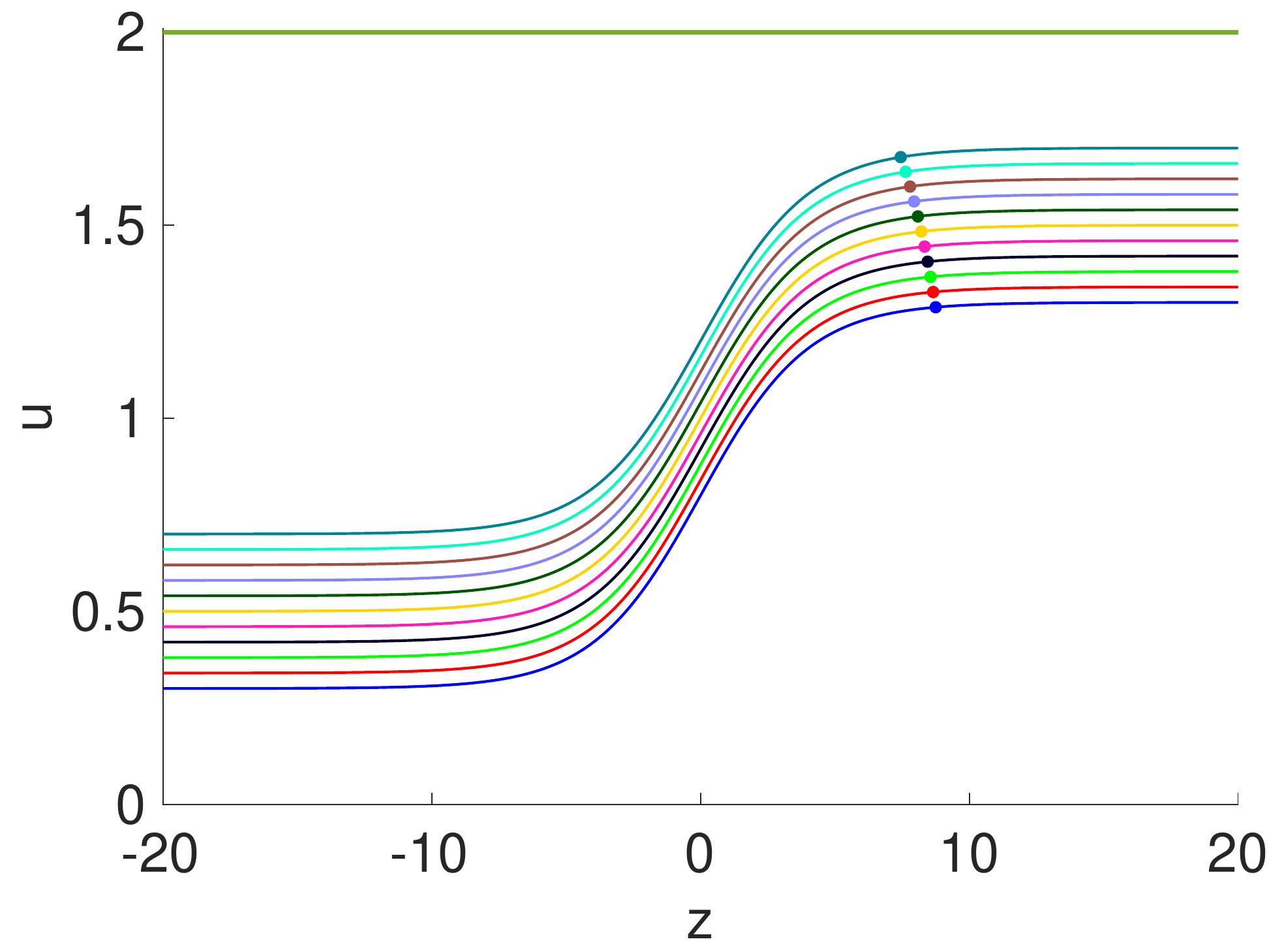}
  \caption{Plot of the constraint $u$ in \eqref{example_opt} as a
    function of $z$ for different values of $\theta$. Left: varying
    $\theta_1$ from $0.3$ to $0.7$ with $\theta_2=0.5$ fixed; right:
    $\theta_2$ varying from $0.3$ to $0.7$ with $\theta_1=0.5$
    fixed. Each curve corresponds to a different $\theta$ and each dot
    corresponds to the solution of \eqref{example_opt} for that given
    $\theta$. The green horizontal line is the target $u=2$.}
  \label{fig:logistic_example}
\end{figure}

\end{example}

In many applications, the parameter space $\Theta$ may be a product of
sets corresponding to different physical parameters. For instance,
there may be parameters corresponding to functions defined on the
PDE's computational domain (such as a spatially distributed
coefficient in the PDE), functions defined on the boundary (such as a
non homogeneous Dirichlet boundary condition), or scalar parameters
(such as constant coefficients in the PDE). It is useful to assign a
scalar measure of sensitivity for each set of parameters. To do so,
consider $\Theta = \Xi_1 \times \Xi_2 \times \cdots \times \Xi_T$ as
the product of $T$ different parameter sets and define the local set
sensitivity index as
\begin{eqnarray}
\label{set_sensitivity}
\mathcal S_{\Xi_i}(\theta_0,z_0) =  \max_{\phi \in \Theta} \left\vert \left\vert \D(\theta_0;z_0) \Pi_{\Xi_i} \frac{\phi}{\vert \vert \phi \vert \vert_\Theta} \right\vert \right\vert_Z \qquad i=1,2,\dots,T,
\end{eqnarray}
where $\Pi_{\Xi_i}: \Theta \to \Theta$ is the projection operator
which maps $\phi=(\xi_1,\xi_2,\dots,\xi_T) \in \Xi_1 \times \Xi_2
\times \cdots \times \Xi_T$ to $\Pi_{\Xi_i}
\phi=(0,0,\dots,0,\xi_i,0,\dots,0) \in \Xi_1 \times \Xi_2 \times
\cdots \times \Xi_T$, i.e. it annihilates all parameter variability
except those in $\Xi_i$. For complex multi-physics problems where $T$
may be large (we consider $T=5$ as large), the set of indices
$\{S_{\Xi_i}(\theta_0,z_0) \}_{i=1}^T$ may be easily tabulated or visualized to see
the relative importance of the different parameter sets. For such
complex problems, \eqref{sensitivity_function} provides detailed
information but it may be difficult to compare the sensitivities of
different parameters if, for instance, they have different time and/or
space dependencies. The set sensitivity indices
\eqref{set_sensitivity} provide a simple summary.

\subsection{Global Sensitivity Analysis}
Caution must be exercised when interpreting local sensitivities
because of their dependence on the nominal parameter $\theta_0$. If
$\theta$ is uncertain and the function $\F$ is nonlinear, then local
sensitivities may not be adequate to make inferences about the
parameters. To address this issue, we introduce a global sensitivity
function. There are two important, and related, points to consider
when defining a global sensitivity function:
\begin{enumerate}
\item[$\bullet$] The term ``global" refers to being global in the parameter space, not in the optimization variables. In general, only local optimality of the optimization problem \eqref{opt_gen} may be ensured, but we may perform analysis at different $\theta_0$'s thus making the analysis global in the parameter space.
\item[$\bullet$] For a fixed $\theta_0$, our analysis is only valid around the local minimum $(u_0,z_0)$. It is necessary to account for different local minima which may occur for a fixed $\theta_0$.
\end{enumerate}

As in classical global sensitivity analysis, let $\mathbf \Theta$ be a
random field (or vector if $\Theta=\R^n$), taking values in $\Theta$,
which encodes the uncertainty in the parameters.  The presence of
multiple local minima poses an additional challenge beyond what is
typically considered in global sensitivity analysis. To formally
define a global sensitivity function, let $\mathbf I$ be a random
field taking values in $U \times Z$. Realizations of $\mathbf I$ are
used as initial iterates for an optimization routine.  Assume that
each $(\theta_0,I_0)$ is a realization of $(\mathbf \Theta, \mathbf
I)$ and is uniquely associated with a particular local minimum
$(u_0(\theta_0,I_0),z_0(\theta_0,I_0))$ of \eqref{opt_gen}. This
ensures that the global sensitivity function introduced below is
well defined. If the initial iterate is set to zero, then $\mathbf I$
will equal zero with probability 1. Otherwise, random initial iterates
can be used in which case different local minima may exist for fixed
parameters and different realizations of $\mathbf I$. The sensitivity indices defined below will depend on the choice of the distribution of $\mathbf I$; however, the presence of multiple local minima is well understood in the optimization community and is frequently managed by leveraging application specific knowledge to choose initial iterates judiciously.

For each $(\theta_0,I_0)$ as a realization of $(\mathbf \Theta,\mathbf
I)$, there is a unique operator $\D(\theta_0;z_0(\theta_0,I_0))$
corresponding to \eqref{sen_operator} when $\theta=\theta_0$ and the
local minimum is identified by solving \eqref{opt_gen} with initial
iterate $I_0 \in U \times Z$. Then considering
\eqref{sensitivity_function} as a function of $(\mathbf \Theta,
\mathbf I)$ yields that $\mathcal S(\mathbf \Theta;z_0(\mathbf
\Theta,\mathbf I))$ is a random field. Assume that $\mathcal S(\mathbf
\Theta;z_0(\mathbf \Theta,\mathbf I)) \phi$ is measurable for each
$\phi \in \Theta$. Mimicking ideas in derivative-based global
sensitivity analysis \cite{kucherenko_derivative, delsa,dgsm1,dgsm2},
we define the global hyper-differential sensitivity function as
$\mathcal S^G:\Theta \to \R$,
\begin{eqnarray}
\label{global_sensitivity_function}
\mathcal S^G \phi = \mathbb E_{\mathbf \Theta,\mathbf I} \left[ \mathcal S(\mathbf \Theta;z_0(\mathbf \Theta,\mathbf I))\phi \right],
\end{eqnarray}
where $\mathbb E_{\mathbf \Theta,\mathbf I} \left[ \cdot \right]$
denotes the expected value computed with respect to the distribution
of $(\mathbf \Theta,\mathbf I)$. We will frequently refer to
\eqref{global_sensitivity_function} as a global sensitivity for
short. Global set sensitivity indices may be defined in a similar
manner by taking the expectation of \eqref{set_sensitivity}, we omit
the details for brevity.

The global sensitivity function $\mathcal  S^G$ may be related to the solution
of the optimization problem \eqref{opt_gen} as follows. Using the same arguments as in the definition of $\mathcal S^G$, we associate each $\theta_0,I_0$ with a unique optimal solution $u_0,z_0,\lambda_0$. The implicit function theorem ensures the existence of a unique function mapping parameters in a neighborhood of $\theta_0$ to optimal solutions in a neighborhood of $u_0,z_0,\lambda_0$, denote it as $\F_{\theta_0,I_0}$ (it was denoted as $\F$ in Subsection~\ref{sec:loc_sen}). Theorem~\ref{sensitivity_opt_thm} shows that the global sensitivity function is related to the average local changes in the optimal solutions.

\begin{theorem}
\label{sensitivity_opt_thm}
Assume that
\begin{enumerate}
\item[$\bullet$] $\phi \in \Theta$ with $\vert \vert \phi \vert \vert_\Theta =1$,
\item[$\bullet$] there exists $\delta > 0$ such that $\delta < \inf\limits_{\theta \in \Theta,I \in U \times Z} r(\theta,I)$, where $r(\theta,I)$ denotes the radius of the neighborhood on which $\F_{\theta,I}$ is defined
\item[$\bullet$] the operators from $\Theta \times U \times Z$ to $Z$ defined by $(\theta,I) \mapsto \P \F_{\theta,I}(\theta+\delta \phi,I)$ and $(\theta,I) \mapsto \P \F_{\theta,I}(\theta,I)$ are measurable
\item[$\bullet$]  each $\P \F_{\theta,I}'$ is Lipschitz continuous with constant $L(\theta,I)$
\item[$\bullet$] $\exists Q \in \R$ such that $Q \ge \sup_{\theta,I} L(\theta,I)$ 
\end{enumerate}
Then
\begin{eqnarray*}
\mathbb E_{\mathbf \Theta,\mathbf I} \left[ \vert \vert \P \F_{\Theta,\mathbf I} (\mathbf \Theta+\delta\phi,\mathbf I) - \P \F_{\Theta,\mathbf I} (\mathbf \Theta,\mathbf I) \vert \vert_Z \right] \le \delta  \mathcal S^G\phi + Q \delta^2.
\end{eqnarray*}
\end{theorem}

\begin{proof}
Let $(\theta_0,I_0)$ be a fixed realization of $(\mathbf \Theta,\mathbf I)$. The mean value inquality implies that
\begin{align*}
 \vert \vert \P \F_{\theta_0,I_0} (\theta_0+\delta\phi,I_0) - \P \F_{\theta_0,I_0}(\theta_0,I_0) \vert \vert_Z \le & \vert \vert  \P \F'_{\theta_0,I_0}(\theta_0+\epsilon \phi,I_0) \delta \phi \vert \vert_Z \\
\end{align*}
where $0 \le \epsilon \le \delta$. Then observe that
\begin{align*}
 \vert \vert  \P \F'_{\theta_0,I_0}(\theta_0+\epsilon \phi,I_0) \delta \phi \vert \vert_Z = & \vert \vert \P \F'_{\theta_0,I_0}(\theta_0,I_0) \delta \phi  +  \P \F'_{\theta_0,I_0}(\theta_0+\epsilon \phi,I_0) \delta \phi - \P \F'_{\theta_0,I_0}(\theta_0,I_0) \delta \phi \vert \vert_Z \\
 & \le \delta \vert \vert \P \F'_{\theta_0,I_0}(\theta_0,I_0) \phi \vert \vert_Z + \delta \vert \vert \P \F'_{\theta_0,I_0}(\theta_0+\epsilon \phi,I_0) - \P \F'_{\theta_0,I_0}(\theta_0,I_0) \vert \vert_Z
\end{align*}
Using the definition of the local sensitivity function and the Lipschitz assumption on $\P \F'_{\theta,I}$ we have
\begin{align*}
 \vert \vert \P \F_{\theta_0,I_0} (\theta_0+\delta\phi,I_0) - \P \F_{\theta_0,I_0}(\theta_0,I_0) \vert \vert_Z \le & \delta \mathcal S(\theta_0,z_0(\theta_0,I_0)) \phi + \delta Q \epsilon.
\end{align*}
Recalling that $\epsilon \le \delta$ and taking the expectation over all $(\theta_0,I_0)$ completes the proof.
\end{proof}

Theorem~\ref{sensitivity_opt_thm} cannot be used to determine computable bounds nor may its assumptions be easily verified in practice; however, it provides a basic intuition about the global sensitivity function by formally connecting its definition back to the optimization problem. The constant $Q$ corresponds to
the nonlinearity of the operator $\F_{\theta,I}$, if $\F_{\theta,I}$ is approximately linear then
$Q$ will be small. Theorem~\ref{sensitivity_opt_thm} implies that the
average change in the optimal solution when the nominal parameters are
perturbed in the direction $\phi$ is bounded by the global sensitivity
function acting on $\phi$, and a measure of nonlinearity. 

Computing the expected value in $\mathcal S^G$ is very difficult in
practice. As elaborated in Section~\ref{sec:postprocessing}, a sparse
sampling approach is taken and the variability of the local
sensitivity function is used as a heuristic to assess the nonlinearity of the parameter to optimal solution mapping. This approach is akin to derivative-based global
sensitivity analysis \cite{kucherenko_derivative, delsa,dgsm1,dgsm2}
and Morris screening
\cite{iooss,morris}. Example~\ref{linearity_example} illustrates a
special case where the parameter to optimal solution mapping is linear.

\begin{example}
\label{linearity_example}
\textit{Consider the following optimization problem}
\begin{align}
& \min_{u,z} \hspace{3 mm} J(u,z)=\frac{1}{2} \vert \vert u-d \vert \vert^2  \\
& s.t. \hspace{5 mm} c(u,z,\theta)=A(\theta)u-z \nonumber
\end{align}
\textit{where $d \in U$ is a target state and $A(\theta)$ is linear
  differential operator and is linear in its parameters. The solution
  of the PDE is given by
$$u(z;\theta)=A^{-1}(\theta)z,$$ which is a nonlinear function of $\theta$. To perform sensitivity analysis of the PDE solution with respect to $\theta$ requires adequate sampling of the parameter space to account for the nonlinearity. On the other hand, the optimal solution is given by 
$$z_{opt}(\theta) = A(\theta)d.$$
Hence the optimal $z$ is a linear function of $\theta$ even through
the PDE solution is a nonlinear function of $\theta$. }

Even though the assumptions in Example~\ref{linearity_example} are too
restrictive to permit any general conclusions, the example illustrates
a case where it is easier to compute sensitivities of the optimal $z$
than sensitivities of the PDE solution.  Local sensitivity analysis of
the optimal solution is therefore sufficient. In general, local
sensitivities should be computed at various samples from parameter
space to assess the nonlinearity.
\end{example}

\section{Computations of Local Sensitivities}
\label{sec:comp_local_sen}

Having formally defined sensitivities, next we present our approach to
efficiently compute local sensitivities. In particular we seek low
rank structure in the parameter space through the SVD of an operator
arising from the solution of the PDE-constrained optimization problem.
Computation of global sensitivities will be considered in
Section~\ref{sec:postprocessing}.

\subsection{Finite Dimensional Approximation}
\label{subsection:Discretization}
The optimization problem \eqref{opt_gen} is discretized by defining
the finite dimensional subspaces
$U_h \subseteq U, Z_h \subseteq Z, \Lambda_h \subseteq \Lambda,
\Theta_h \subseteq \Theta$. In practice these subspaces typically
arise from a discretization of the PDE, for instance, a finite element
discretization. Let $\theta_0 \in \Theta_h$ be the nominal parameters,
and $(u_0, z_0,\lambda_0) \in U_h \times Z_h \times \Lambda_h$ be a
local minimum of the discretization of \eqref{opt_gen} with nominal
parameters $\theta_0$.

Letting $\{\phi_1,\phi_2,\dots,\phi_n\}$ be a basis for $\Theta_h$, we
define the local hyper-differential sensitivity indices as
\begin{eqnarray}
\label{local_sensitivity_index}
S_i(\theta_0,z_0) = \mathcal S(\theta_0,z_0)\phi_i, \qquad i=1,2,\dots,n.
\end{eqnarray}
Global hyper-differential sensitivity indices $S_i^G$,
$i=1,2,\dots,n$, may be defined in a similar manner by having
\eqref{global_sensitivity_function} act on the basis functions. For
simplicity we assume that $\vert \vert \phi_i \vert \vert_\Theta=1$ for
$i=1,2,\dots,n$.

The indices $S_i(\theta_0,z_0)$, $i=1,2,\dots,n$, may be computed
directly by solving \eqref{opt_gen} once, and subsequently computing
each $S_i(\theta_0,z_0)$ separately. This approach requires solving
$n$ linear systems with coefficient matrix $\K$ (or solving a block
system with $n$ right hand sides). However, in many applications $\B$
(in \eqref{K_B_ops}) possesses a low rank structure which may be
exploited to accelerate computation.

Assume that $\D(\theta_0;z_0)$ is a compact operator. Letting
$\sigma_k,\theta_k,z_k$, $k=1,2,\dots$ denote the singular values and
vectors of $\D(\theta_0;z_0)$,
i.e. $\D(\theta_0;z_0)\theta_k=\sigma_kz_k$, we have
\begin{eqnarray}
\label{sensitivity_svd_representation}
S_i(\theta_0,z_0) = \sqrt{ \sum\limits_{k=1}^\infty \sigma_k^2 (\theta_k,\phi_i)_\Theta^2 }, \qquad i=1,2,\dots,n,
\end{eqnarray}
where $(\cdot,\cdot)_\Theta$ denotes the inner product on the Hilbert space $\Theta$.

Truncating the series in \eqref{sensitivity_svd_representation} yields the approximations
\begin{eqnarray*}
S_i(\theta_0,z_0) \approx \sqrt{ \sum\limits_{k=1}^K \sigma_k^2 (\theta_k,\phi_i)_\Theta^2 }, \qquad i=1,2,\dots,n.
\end{eqnarray*}
The singular values/vectors $\sigma_k,\theta_k$, $k=1,2,\dots,K$ may
frequently be computed with far fewer than $n$ applications of the
operator $\D(\theta_0;z_0)$. For many problems in practice, $n$ is on
the order of hundreds, thousands, or more, whereas $K$ on the order of
tens may be sufficient to accurately approximate $S_i$,
$i=1,2,\dots,n$. Leveraging the truncated SVD representation of
$\D(\theta_0;z_0)$ may reduce the number of linear system solves by an
order of magnitude (or more) if such low rank structure exists. The
presence (or lack thereof) of low rank structure is easily identified
by computing the leading singular values of $\D(\theta_0;z_0)$, so the
approximation is easily certified in practice.

In addition to facilitating efficient estimation of the sensitivity
indices, the truncated SVD also provides directions of greatest
sensitivity in parameter space (the right singular vectors). These
directions provide an alternative coordinate system in parameter space
with gives a useful low dimensional representation. Though different
in several ways, these singular vectors and sensitivity indices have a
similar intuition as active subspaces \cite{active_subspaces} and
activity scores \cite{activity_scores}, respectively. Additionally,
the left singular vectors in the space $Z$ indicate which features of the optimal
solution (in space and/or time) are most sensitivity to the
parameters.

Assuming that the singular values/vectors $\sigma_k,\theta_k, z_k$,
$k=1,2,\dots,K$ have been computed, the local set sensitivity index
\eqref{set_sensitivity} may be estimated by computing the largest
singular value of the linear operator
\begin{eqnarray*}
\phi \mapsto \sum\limits_{k=1}^K \sigma_k z_k (\theta_k,\Pi_{\Xi_i} \phi)_\Theta.
\end{eqnarray*}
For low rank operators (small $K$), the computational cost of this
estimate is negligible in comparison to computing $\sigma_k,\theta_k,
z_k$, $k=1,2,\dots,K$.

For problems where $n$ is large and $\D(\theta_0;z_0)$ does not admit
a low rank structure, computing the local sensitivity indices
\eqref{local_sensitivity_index} is prohibitive. In such cases, the
local set sensitivity indices \eqref{set_sensitivity} may be estimated
by using an SVD routine to compute the leading singular value of
$\D(\theta_0;z_0) \Pi_{\Xi_i}$ for each $i=1,2,\dots,T$. This is a
significant saving when, for instance, $T = \mathcal O(5)$ and $n =
\mathcal O(1,000)$, which occurs when there are multiple temporally
and/or spatially dependent parameters.

\subsection{Computing the Truncated SVD}

The operator $\D$ is defined on function spaces so the truncated SVD
  should be computed using the inner products from $\Theta_h$ and
  $Z_h$. To achieve this, an SVD routine using Euclidean inner
  products may be applied to the matrix
\begin{eqnarray}
\label{computation_op}
R_{Z_h} D R_{\Theta_h}^{-1} ,
\end{eqnarray}
where $R_{\Theta_h}$ and $R_{Z_h}$ are the Cholesky factors of the
symmetric positive definite mass matrices (or weighting matrices more
generally)
$M_{\Theta_h} = R_{\Theta_h}^T R_{\Theta_h} \in \R^{n \times n}$ and
$M_{Z_h} = R_{Z_h}^T R_{Z_h} \in \R^{m \times m}$ defined by
\begin{eqnarray*}
(M_{\Theta_h})_{i,j} = (\phi_i,\phi_j)_{\Theta} \qquad (M_{Z_h})_{i,j} = (y_i,y_j)_{Z},
\end{eqnarray*}
where $\{y_1,y_2,\dots,y_m\}$ is a basis for $Z_h$.

The matrix $D=C_{Z_h} \D E_{\Theta_h} \in \R^{m \times n}$ represents
the action of $\D$ on coordinates in the discretized spaces, where
$E_{\Theta_h}: \R^n \to \Theta_h \text{ and } C_{Z_h}: Z_h \to \R^m$
denote the coordinate transformation operators.  Computing the
truncated SVD of \eqref{computation_op} directly is not scalable
because the Cholesky factors $R_{Z_h}$ and $R_{\Theta_h}$ will be
dense whereas the mass matrices are typically sparse. In what follows,
we propose to a symmetric generalized eigenvalue problem instead, and
mitigate computational limitations by introducing the use of a
randomized generalized eigenvalue solver from \cite{arvind} which
facilitates parallel evaluations of matrix vector products.

\subsection{Proposed Reformulation}
\label{sec:reformulation}
For notational simplicity, $\theta$ and $z$ are used to denote the
coordinate representations for parameters and optimization variables
throughout this section. The singular values and singular vectors of
\eqref{computation_op} correspond to the positive eigenvalues and
eigenvectors of the Jordan-Wielandt matrix
\begin{eqnarray*}
 W= \left( \begin{array}{cc}
0 & R_{Z_h} D R_{\Theta_h}^{-1} \\
(R_{\Theta_h}^{-1})^T D^T (R_{Z_h})^T & 0 \\
\end{array} \right) .
\end{eqnarray*}

This gives the symmetric eigenvalue problem
\begin{eqnarray}
\label{eig_problem}
W  \left( \begin{array}{c}
z\\
\theta\\
\end{array} \right)
= \alpha
 \left( \begin{array}{c}
z\\
\theta\\
\end{array} \right)
\end{eqnarray}
for which we seek to compute the largest eigenvalues.

To avoid computing the Cholesky factors of the mass matrices, define
\begin{eqnarray*}
X= \left( \begin{array}{cc}
R_{Z_h}^{-1} & 0 \\
0 & R_{\Theta_h}^{-1} \\
\end{array} \right)
\qquad \text{and} \qquad
 \left( \begin{array}{c}
\tilde{z}\\
\tilde{\theta}\\
\end{array} \right)
= X
\left( \begin{array}{c}
z\\
\theta\\
\end{array} \right) .
\end{eqnarray*}
Then we may reformulate \eqref{eig_problem} as
\begin{eqnarray}
\label{gen_eigen}
A  \left( \begin{array}{c}
\tilde{z}\\
\tilde{\theta}\\
\end{array} \right)
 = \alpha B  \left( \begin{array}{c}
\tilde{z}\\
\tilde{\theta}\\
\end{array} \right).
\end{eqnarray}
where
\begin{eqnarray*}
A=(X^{-1})^T W X^{-1}
= \left( \begin{array}{cc}
0 & M_{Z_h} D \\
 D^T M_{Z_h} & 0 \\
\end{array} \right)
\qquad
\text{and}
\qquad
B=(X^{-1})^T X^{-1} 
= \left( \begin{array}{cc}
M_{Z_h} & 0 \\
0 & M_{\Theta_h} \\
\end{array} \right),
\end{eqnarray*}

Since $A$ is symmetric and $B$ is symmetric positive definite, a
symmetric generalized eigenvalue problem must be solved. Symmetry in
our formulation ensures good numerical properties and leverages
powerful theoretical results from linear algebra
\cite{structured_eigen}. We propose to solve the generalized
eigenvalue problem with a randomized algorithm, the advantages of this
approach will be elaborated on in Section~\ref{sec:algo_overview}.

Assume that the $K$ largest eigenvalues $\alpha_k > 0 $ and
corresponding eigenvectors $(\tilde{z}_k,\tilde{\theta_k})$,
$k=1,2,\dots,K$, of \eqref{gen_eigen} have been computed. The
corresponding singular values, right singular vectors, and left
singular vectors of \eqref{computation_op} (using Euclidean inner
products) are $\alpha_k$, $R_{\Theta_h} \tilde{\theta}_k$, and
$R_{Z_h} \tilde{z}_k$, $k=1,2,\dots,K$, respectively. It appears that
the mass matrix Cholesky factor is needed to compute these singular
vectors; however, as shown in \cite{griesse2}, the coordinate
representation of the right and left singular vectors of
\eqref{sen_operator} (using $\Theta_h$ and $Z_h$ inner products) are
given by
\begin{eqnarray*}
\frac{\tilde{\theta}_k}{\sqrt{ \tilde{\theta}_k^T M_{\Theta_h} \tilde{\theta}_k } }
\qquad
\mbox{and}
\qquad
 \frac{\tilde{z}_k}{\sqrt{ \tilde{z}_k^T M_{Z_h} \tilde{z}_k } },
\end{eqnarray*}
  $k=1,2,\dots,K$, respectively. Hence we only compute the largest
eigenvalues and eigenvectors of \eqref{gen_eigen}, and normalize them
with the mass matrices.
  
An alternative formulation solves $D^T M_{Z_h} D \theta = \alpha
M_{\Theta_h} \theta$ instead of \eqref{gen_eigen}. This system is $n
\times n$ positive definite, instead of $(n+m) \times (n+m)$
indefinite, and squares the singular values, which is advantageous for
computing singular values greater than one. However, this formulation
requires an additional collection of $K$ matrix vector products to
compute the left singular vectors $z_k$, $k=1,2,\dots,K$. Both
formulations were considered, the results in this article focus on the
generalized eigenvalue problem \eqref{gen_eigen} because of its ease
computing the left singular vectors.

\section{Algorithmic Overview}
\label{sec:algo_overview}
Algorithm~\ref{alg:hyperdiff} below provides an overview of
hyper-differential sensitivity analysis and highlights the important
computational features.  The solution to the underlying
PDE-constrained problem is encapsulated in Line 3. A core component of
Algorithm~\ref{alg:hyperdiff} is the randomized algorithm, in Lines
4-12, used to solve the generalized eigenvalue problem. We begin by
motivating the randomized generalized eigenvalue solver and
subsequently consider the computational complexity of
Algorithm~\ref{alg:hyperdiff}.

\begin{algorithm}
\caption{Hyper-differential Sensitivity Analysis Algorithm}
\label{alg:hyperdiff}
\textbf{Input: } number of parameter samples $N$, number of singular pairs $K$, oversampling factor $L$\\
\textbf{1: } for $j$ from 1 to $N$ (embarrassingly parallel loop) \\
\textbf{2: } \hspace{1 cm} sample $(\overline{\theta}^j,\overline{I}^j)$ from the distribution of $(\mathbf \Theta,\mathbf I)$\\
\textbf{3: } \hspace{1 cm} solve \eqref{opt_gen} with $\theta=\overline{\theta}^j$ and initial iterate $\overline{I}^j$ and store solution $u_{opt}^j,z_{opt}^j$\\
\textbf{4: } \hspace{1 cm} for $i$ from 1 to $2K+L$ (embarrassingly parallel loop)\\
\textbf{5: } \hspace{2 cm} draw a standard normal sample 
$\left( \begin{array}{c}
\tilde{z}^i \\
\tilde{\theta}^i \\
\end{array} \right)$
and compute $y_i = B^{-1} A
\left( \begin{array}{c}
\tilde{z}^i \\
\tilde{\theta}^i \\
\end{array} \right) \in \R^{m+n}$\\
\textbf{6: } \hspace{1 cm} end\\
\textbf{7: } \hspace{1 cm} compute the decomposition $QR=[y_1,y_2,\dots,y_{2K+L}] \in \R^{(m+n)\times (2K+L)}$ with $B$ inner products\\
\textbf{8: } \hspace{1 cm} for $i$ from 1 to $2K+L$ (embarrassingly parallel loop) \\
\textbf{9: } \hspace{2 cm} compute $Aq_i \in \R^{(m+n)}$ where $Q=[q_1,q_2,\dots,q_{2K+L}] \in \R^{(m+n) \times (2K+L)}$\\
\textbf{10:} \hspace{1 cm} end\\
\textbf{11:} \hspace{1 cm}  form $T=Q^T A Q \in \R^{(2K+L) \times (2K+L)}$\\
\textbf{12:} \hspace{1 cm}  compute the eigenvalue decomposition $T=VEV^T$ with ordering $E_{1,1} \ge E_{2,2} \ge \cdots \ge E_{2K+L}$ \\
\textbf{13:} \hspace{1 cm} form
$\left( \begin{array}{c}
\tilde{z}_k \\
\tilde{\theta}_k \\
\end{array} \right)=Qv_k \in \R^{(m+n)}$, $k=1,2,\dots,K$, where $V=[v_1,v_2,\dots,v_{2K+L}]$ \\
\textbf{14:} \hspace{1 cm} store singular values $\sigma_k^j = E_{k,k}$, $k=1,2,\dots,K$\\
\textbf{15:} \hspace{1 cm} compute and store singular vectors $z_k^j=\frac{\tilde{z}_k}{\tilde{z}_k^T M_{Z_h} \tilde{z}_k}$ and $\theta_k^j=\frac{\tilde{\theta}_k}{\tilde{\theta}_k^T M_{\Theta_h} \tilde{\theta}_k}$, $k=1,2,\dots,K$\\
\textbf{16:} end \\
\textbf{Return: } $u_{opt}^j, z_{opt}^j$, $\sigma_k^j$, $z_k^j,\theta_k^j$, $k=1,2,\dots,K$, $j=1,2,\dots,N$\\
\end{algorithm}

\subsection{Randomized Linear Algebra}
Randomized linear algebra has emerged as a powerful tool in scientific
computation \cite{randomized_la_review}. The utility of randomized
methods is that they permit a reordering of the computation which may
better exploit computing architectures. Most traditional algorithms
are inherently serial, for instance, constructing Krylov subspaces
require serial matrix-vector products since each vector is formed
using the previous ones. In contrast, randomized methods may require a
comparable number of matrix-vector products which can be computed in
parallel with minimal communication overhead.

We adopt the randomized generalized eigenvalue solver from
\cite{arvind}. It is well suited for estimating the largest
eigenvalues. The user specifies the desired number of eigenvalues $p$
and an oversampling factor $L$. Then $p+L$ matrix vector products are
computed (in parallel) by applying the coefficient matrix to
independently generated random vectors. The resulting vectors form an
approximation of the subspace of eigenvectors corresponding to the
largest eigenvalues. Inexpensive computation may be done in this
subspace to estimate the eigenvalues and eigenvectors.

Algorithm~\ref{alg:hyperdiff} inputs an integer $N$ specifying the
number of local sensitivities to compute, an integer $K$ specifying
the number of singular pairs to compute for each local sensitivity,
and an integer $L$ specifying the oversampling factor. Large values of
$N$ may be necessary to accurately estimate global sensitivities;
however, relatively small values of $N$, for instance $\mathcal
O(10)$, are frequently sufficient to capture important features. The
user should choose $N$ based on their available computational
resources. Ideally, the $N$ local sensitivities will provide similar
parameter inferences. If the local sensitivities differ significantly
(the operator $\F$ is highly nonlinear) then the user should exercise
caution making inferences with them. The ``optimal" choice for $K$ is
not clear a-priori; however, the singular values returned from
Algorithm~\ref{alg:hyperdiff} certify the choice of $K$ (by assessing
the low rank structure). The user should start with small values for
$K$, for instance $K=4$ is used in
Section~\ref{sec:numerical_results}, and increase it if necessary. If
there is not sufficient decay in the first $K$ singular values then
Lines 4-12 may be repeated to augment the existing computation. The
oversampling factor $L$ scales the cost versus accuracy. Typically $L
< 20$ (or even $L<10$) is sufficient, see \cite{randomized_la_review,
  arvind} and references therein. In Algorithm~\ref{alg:hyperdiff},
$2K+L$ random vectors are used (see Lines 4-6) because the eigenvalues
of $A$ correspond to positive and negative pairs of the desired
singular values.

The outer loop initialized in Line 1 of Algorithm~\ref{alg:hyperdiff}
is embarrassingly parallel. The most computationally intensive
portions within this loop are Line 3 (solving the PDE-constrained
optimization problem), Line 5 (applying $A$), and Line 9 (applying
$A$). All of the other calculations are simple linear algebra which
may be executed quickly using standard libraries. Solving the
PDE-constrained optimization problem in Line 3 involves a host of
complexities including trust region and/or line search globalization,
finite element discretization, solving large systems of (possibly
nonlinear) equations, and adjoint calculations to evaluate gradients
and/or hessians. In the scope of this work, we assume that efficient
solvers are available for this end but emphasize its complexity. Lines
5 and 9 are computationally intensive because applying $A$ requires
applying $\K^{-1}$ twice, which involves potentially many PDE solves.

The benefit of the randomized generalized eigenvalue solver is that
the $2K+L$ application of $A$ in Lines 5 and 9 may be
parallelized. The loops are embarrassingly parallel since the random
vectors are independent; however, barriers are necessary after each
loop because the data from each matrix vector product must be shared
across processors in order to form $Q$ and $T$ (Lines 7 and 11). 

\subsection{Algorithmic Complexity}

To asses the computational complexity of Algorithm~\ref{alg:hyperdiff}
we count the number of large scale linear system solves. For
simplicity our assessment focuses on Lines 4-15 as the computational
complexity of Line 3 is a question of PDE-constrained optimization and
the outer loop initialized at Line 1 simply scales the cost of Lines
2-15 by a factor $N$. The computational cost of Lines 4-15 is
approximately equal to the cost of applying $\mathcal K^{-1}$ $4(2K+L)$
times (twice in Lines 5 and 9). Assuming that each application of
$\mathcal K^{-1}$ requires an average of $s_{CG}$ iterations of
conjugate gradient, this amounts to $4(2K+L)s_{CG}$ times the cost of each
matrix vector product $\mathcal K v$. Since $\mathcal K$ is evaluated
at the optimal solution we may use the existing solution for the
primal and adjoint PDE solves, so computing $\mathcal K v$ requires a
state sensitivity solve (linear system involving the state Jacobian of
the constraint) and a adjoint sensitivity solve (linear system
involving the state Jacobian of the constraint transposed). The total
computational cost will be approximately $4(2K+L)s_{CG}$ state
sensitivity solves plus $4(2K+L)s_{CG}$ adjoint sensitivity
solves. 

The factor of $(2K+L)$ may be mitigated by the randomized solver. The loops in Lines 4-6 and Lines 8-10 are embarrassingly parallel. They have the potential to attain perfect parallel efficiency; however, there is a requirement of synchronization (sharing of data across processors) after Line 6 and 10 which may prevent them from attaining it. Assuming that $2K+L$ or more processors are used, and denoting the wall clock time to execute Line 5 as $T_{wall}(i)$, $i=1,2,\dots,2K+L$, then the wall clock time for Lines 4-6 will be $\max_{i=1,2,\dots,2K+L} T_{wall}(i)$. Perfect parallel efficiency occurs when $T_{wall}(1)=T_{wall}(2)=\cdots=T_{wall}(2K+L)$. The parallel efficiency is difficult to analyze in general because $T_{wall}(i)$ is primarily determined by the number of linear solver iterations required to invert $\mathcal K$, which depends upon the spectral characteristic of $\mathcal K$ and the right hand side of the linear system, which depends on properties of $\mathcal B$, $M_{Z_h}$, and the random vectors sampled in Line 5. A similar wall clock time analysis applies to the loop in Lines 8-10. 

If more than $2K+L$ processors are used, then the PDE solves required in Lines 5 and 9 may be parallelized as well. This parallelization has greater communication requirements and its efficiency will depends upon many factors, most notably the number of degrees of freedom in the discretized PDEs. We defer further discussion on this parallelism as it depends on many implementation specific factors and is beyond the scope of this article. 

\subsection{Interpretation}
\label{sec:postprocessing}

Estimating \eqref{global_sensitivity_function} accurately through
sampling based approaches may require extensive computational effort
(a large $N$ in Algorithm~\ref{alg:hyperdiff}). However, useful
information may be obtained by computing local sensitivities at a
sparse collection of samples from the parameter space. This section
details practical considerations computing, visualizing, and
interpreting samples of local sensitivities. The proposed approach is
demonstrated in Section~\ref{sec:numerical_results}.

Assume that Algorithm~\ref{alg:hyperdiff} has been executed yielding
optimal solutions and local sensitivities at $N$ (chosen based upon
the user's computational budget) different realizations of $(\mathbf
\Theta,\mathbf I)$. For each $j=1,2,\dots,N$, we have,
\begin{enumerate}
\item[$\bullet$] the state $u_{opt}^j$ and optimization variables
  $z_{opt}^j$ determined by solving \eqref{opt_gen} with initial
  iterate $\overline{I}^j$ when $\theta = \overline{\theta}^j$,
\item[$\bullet$] the $K$ leading singular triples of
  \eqref{sen_operator}, $(\sigma_k^j,\theta_k^j,z_k^j)$,
  $k=1,2,\dots,K$, in coordinate representation.
\end{enumerate}
The spectrum of \eqref{sen_operator} summarizes its low rank structure
(or lack thereof). The singular values $\sigma_k^j$, $k=1,2,\dots,K$,
$j=1,2,\dots,N$, may be visualized in a scatter plot which readily
reveals the structure of the leading singular values for different
samples in parameter space. Ideally we will see a decay in the
singular values for each $j$, indicating a low rank structure. If such
a low rank structure exists, the $K$ singular values and singular
vectors may be used to analyze the sensitivities. Otherwise, a larger
value of $K$ is needed. Our approach is designed to exploit low rank
structure which we are implicitly assuming is present.

The local sensitivity index \eqref{sensitivity_svd_representation} is
approximated for the $i^{th}$ parameter basis function using the
$j^{th}$ sample with
\begin{eqnarray}
\hat{S}_i^j = \sqrt{\sum_{k=1}^K (\sigma_k^j)^2  \left( \left(M_{\Theta_h}\theta_k^j\right)_i \right)^2 } \label{sensitivity_index}
\end{eqnarray}
for $i=1,2,\dots,m$, $j=1,2,\dots,N$;
$\left(M_{\Theta_h}\theta_k^j\right)_i$ denotes the $i^{th}$ entry of
the vector $M_{\Theta_h}\theta_k^j$. Then $\hat{S}_i^j$,
$i=1,2,\dots,m$, $j=1,2,\dots,N$, may be visualized in a scatter plot
which reveals the relative influence of the parameters and the
variability of their local sensitivity indices over the samples
$\overline{\theta}^j, \overline{I}^j$, $j=1,2,\dots,N$. This
variability provides a heuristic to measure the nonlinearity of the
parameter to optimal solution mapping. We hope to find a similar low
rank structure and local sensitivity indices for each parameter sample
which indicates desirable structure in the problem. If the local
sensitivities vary significantly over the (sparse) sampling of
parameter space, this indicates strong nonlinearities which will
mandate greater computational effort.

As discussed in Subsection~\ref{subsection:Discretization}, the local
set sensitivity indices may be estimated as a by-product using the
singular values and vectors. They may be easily visualized in a manner
similar to the local sensitivity indices. This is particularly useful
when the visualization of the local sensitivity indices becomes
cumbersome because of temporal and/or spatial dependencies of the
parameters.

Along with the spectrum and parameter sensitivity information above,
we may also visualize the resulting changes in the optimal solution
when the parameters are perturbed. In particular, $z_k^j$ is the
change in the optimal solution if the parameter vector
$\overline{\theta}^j$ is perturbed in the direction $\theta_k^j$. The
singular vectors $z_k^j$ may be visualized by overlaying them on a
plot, or plotting statistical quantities computed from the sample
$\{z_k^j\}_{j=1}^N$. In both cases, it informs us which features of
the optimal solution are most sensitive to the parametric uncertainty.

\section{Numerical Results}
\label{sec:numerical_results}
In this section we demonstrate our proposed hyper-differential
sensitivity analysis on two examples. The first is a control problem
for a steady state nonlinear multi-physics system modeling a high
pressure chemical vapor deposition (CVD) reactors, see
\cite{cvd,Kouri2018}. The second example is an inverse problem for a
transient multi-physics system modeling subsurface contaminant
transport.

Our HDSA software is implemented in the Rapid Optimization Library
(ROL) \cite{rol} of Trilinos \cite{Trilinos-Overview}, a collection of
C++ libraries for scientific computation.  The implementation is based
on C++, Trilinos parallel constructs, and special PDE-constrained
solver interfaces that generalize the solution procedure to a range of
PDE-based models.  ROL consists of state-of-the-art Newton Krylov
based optimization methods with both reduced and full space solution
methods (trust region and line search globalization). The HDSA
implementation is matrix free, has three levels of parallelism (two
embarrassingly parallel loops and parallel linear algebra constructs),
and may be easily adapted to a variety of applications. The parallel
and matrix free design achieves scalable performance.

\subsection{Control of Thermal Fluids}
In this subsection we consider control of the steady state nonlinear
multi-physics Boussinesq flow equations in two spatial dimensions, a
model for a CVD reactor. Reactant gases are injected in the top of a
reactor and flow downwards to create an epitaxial film on the
bottom. Vorticities created by buoyancy-driven convection inhibit some
gases from reaching the bottom of the reactor. Thermal fluxes are
controlled on the side walls of the reactor in order to minimize the
vorticity. Formally, consider the control problem,

\begin{figure}[h]
\centering
  \includegraphics[width=0.49\textwidth]{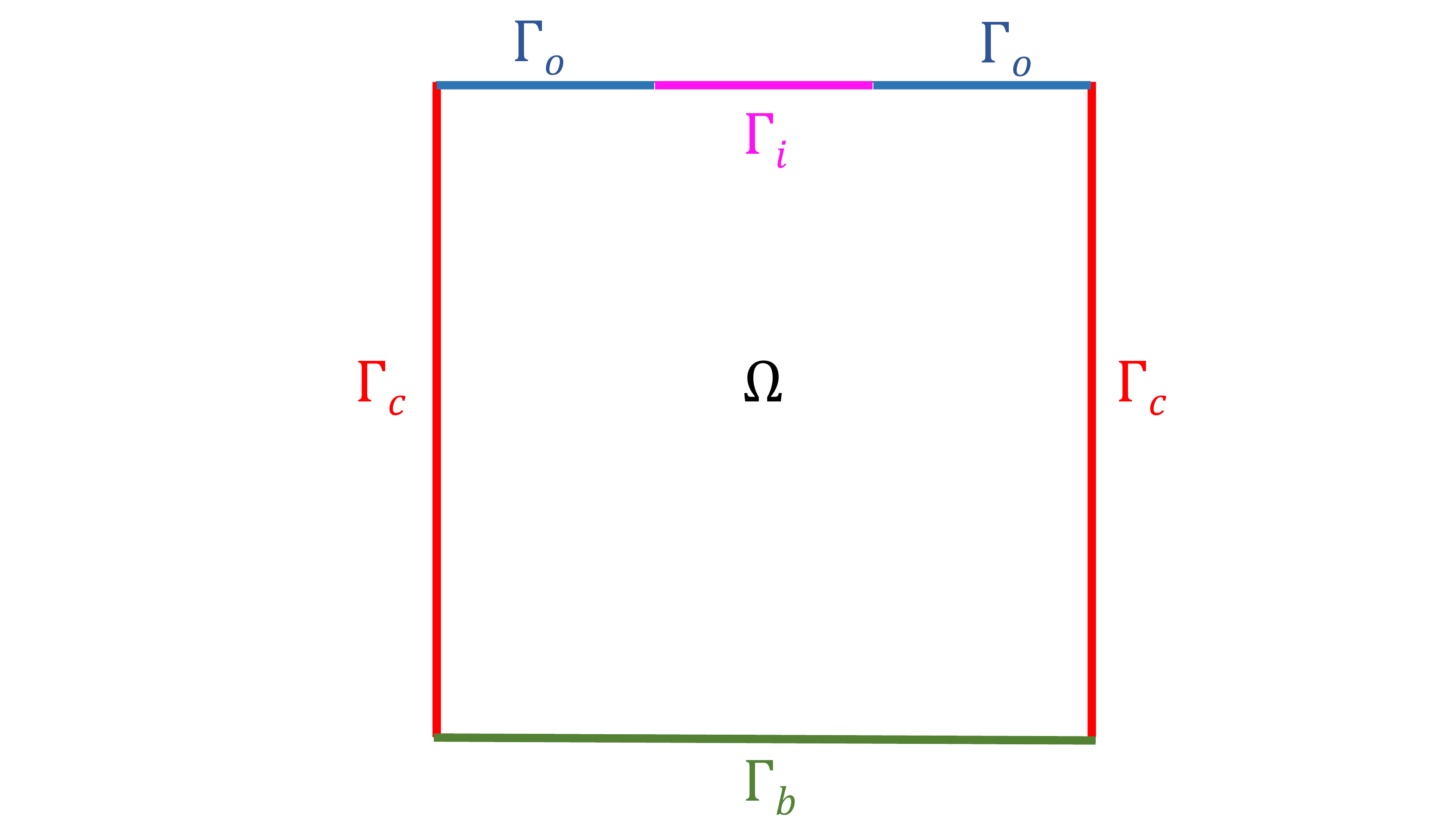}
  \caption{Computation domain and boundaries for \eqref{opt_boussinesq}.}
  \vspace{-.5 cm}
  \label{fig:domain}
\end{figure}

\begin{align} 
& \min\limits_{v,p,T,z} \frac{1}{2} \int_{\Omega} (\nabla \times v)^2 d x + \frac{\gamma}{2} \int_{\Gamma_c} z^2 d x \label{opt_boussinesq} \\
& s.t. \nonumber \\
& -\epsilon(\theta) \nabla^2 v + (v \cdot \nabla ) v + \nabla p + \eta(\theta) T g = 0 & \text{ in } \Omega  \nonumber \\
& \nabla \cdot v = 0 & \text{ in } \Omega  \nonumber \\
& - \kappa(\theta) \Delta T + v \cdot \nabla T = 0 & \text{ in } \Omega  \nonumber \\
&T = 0  \qquad \text{and} \qquad v = v_i & \text{on } \Gamma_i  \nonumber \\
& \kappa(\theta) \nabla T \cdot n= 0 \qquad \text{and} \qquad v=v_o & \text{on } \Gamma_o  \nonumber \\
& T = T_b(\theta)  \qquad \text{and} \qquad v=0 & \text{on } \Gamma_b  \nonumber \\
&  \kappa(\theta) \nabla T \cdot n + \nu(\theta) (z-T) = 0 \qquad \text{and} \qquad v=0 & \text{on } \Gamma_c  \nonumber 
\end{align}
where $\Omega = (0,1) \times (0,1)$, $\Gamma_i = [1/3,2/3] \times
\{1\}$, $\Gamma_o = [0,1/3] \times \{1\} \cup [2/3,1] \times \{1\}$,
$\Gamma_b = [0,1] \times \{0\}$, $\Gamma_c = \{0,1\} \times [0,1]$,
and $n$ denotes the outward pointing normal vector to the
boundary. Figure~\ref{fig:domain} depicts the domain and boundaries.

The state consists of horizontal and vertical velocities
$v=(v_1,v_2)$, the pressure $p$, and the temperature $T$; the control
$z$ is a function defined on the left and right boundaries of
$\Omega$, denote their union $\Gamma_c$. The deterministic inflow and
outflow conditions $v_i$ and $v_o$ are given by
\begin{eqnarray*}
v_i(x) = -4\left(x-\frac{1}{3}\right)\left(\frac{2}{3}-x\right)
\end{eqnarray*}
and
\begin{align*}
v_o(x) =
 \left\{\begin{array}{ll}  2\left(\frac{1}{3}-x\right)x & \mbox{ if } x \in \left[0,\frac{1}{3}\right] \vspace{.2 cm} \\
    2\left(x-\frac{2}{3}\right)\left(1-x\right)  & \mbox{ if } x \in \left[\frac{2}{3},1\right]
  \end{array} \right. .
\end{align*}

Uncertainties enter the the model through the finite dimensional
vector (coming from a spatial discretization) $\theta \in \R^m$ where
$m=2m_b + 2m_\ell + 2m_r + 3$. In particular, the boundary term $T_b$
is defined through the sum
\begin{eqnarray}
\label{T_b_series}
T_b(x,\theta) = 1+0.2 \sum\limits_{k=1}^{m_b} \left( \theta_{k}
\frac{\sin(\pi k x)}{k} + \theta_{m_b+k} \frac{\cos(\pi k x)}{k}
\right).
\end{eqnarray}
In addition, the boundary term $\nu(\theta)$, a function defined on
$\Gamma_c$, is expressed as a sum in the form of \eqref{T_b_series}
with $2m_\ell$ parameters in the sum defining the left boundary term
and $2m_r$ parameters in the sum defining the right boundary
term. Along with these $2m_b + 2m_\ell + 2m_r$ parameters which define
$T_b$ and $\nu$, there are three parameters, $\theta_{m-2},
\theta_{m-1}, \theta_m$, which appear in the uncertain scalar
quantities $\epsilon, \eta,$ and $\kappa$, as
\begin{align*}
& \epsilon = \frac{1}{Re}=\frac{1+0.05 \theta_{m-2}}{100}, \qquad \kappa = \frac{1}{RePr} = \frac{1+0.05 \theta_{m-1}}{72}\\ 
& \text{and } \eta = \frac{Ge}{Re^2} = 1+0.05\theta_m, 
\end{align*}
where $Re$, $Ge$, and $Pr$ are the Reynolds, Grashof,
and Prandtl numbers respectively. Each $\theta_k$, $k=1,2,\dots,m$, is
assumed to be independent and uniformly distributed on $[-1,1]$. The
initial iterate is taken to be zero for all samples; similar results
were found using a random initial iterate. The parameter weighting
matrix $M_p$ is the identity since the uncertain boundary conditions
$T_b$ and $\nu$ are discretized by orthogonal global basis functions.

The PDE is discretized with finite elements on a 99x99 rectangular
mesh. The velocity and pressure are represented with the Q2-Q1
Taylor-Hood finite element pair and the temperature is represented
with the Q2 finite element. We take $m_b=m_\ell=m_r=25$ which yields a
total of $m=153$ uncertain parameters. The deterministic control
problem is solved using the full space composite step algorithm in ROL
with the control penalty
$\gamma=0.01$. Figure~\ref{fig:velocity_field} displays the
uncontrolled (left) and controlled (right) velocity field with
parameters $\theta_k=0$, $k=1,2,\dots,153$. The undesired vorticities
are observed in the uncontrolled velocity field and are reduced by the
control strategy.

\begin{figure}
\centering
  \includegraphics[width=0.4\textwidth]{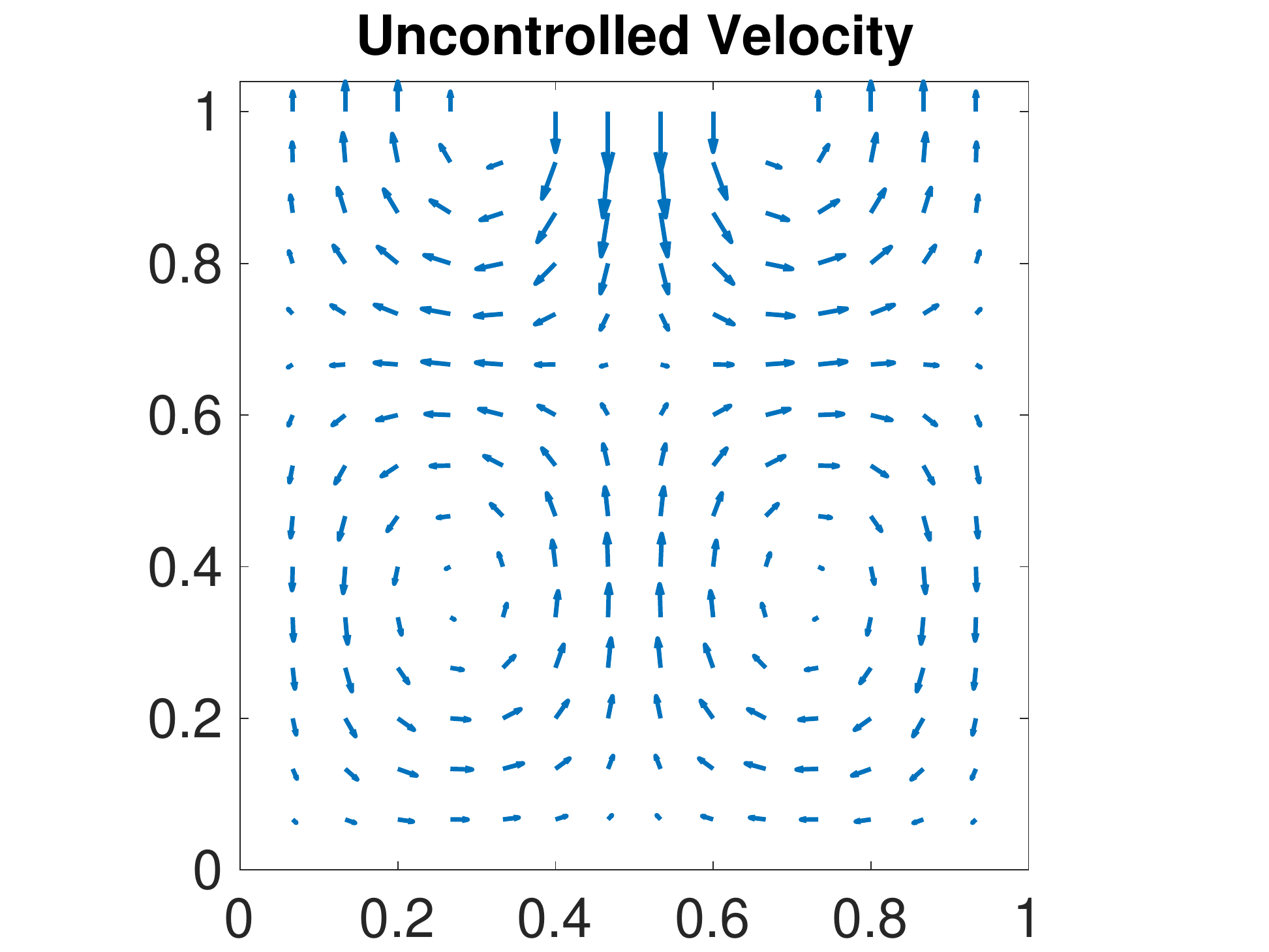}
    \includegraphics[width=0.4\textwidth]{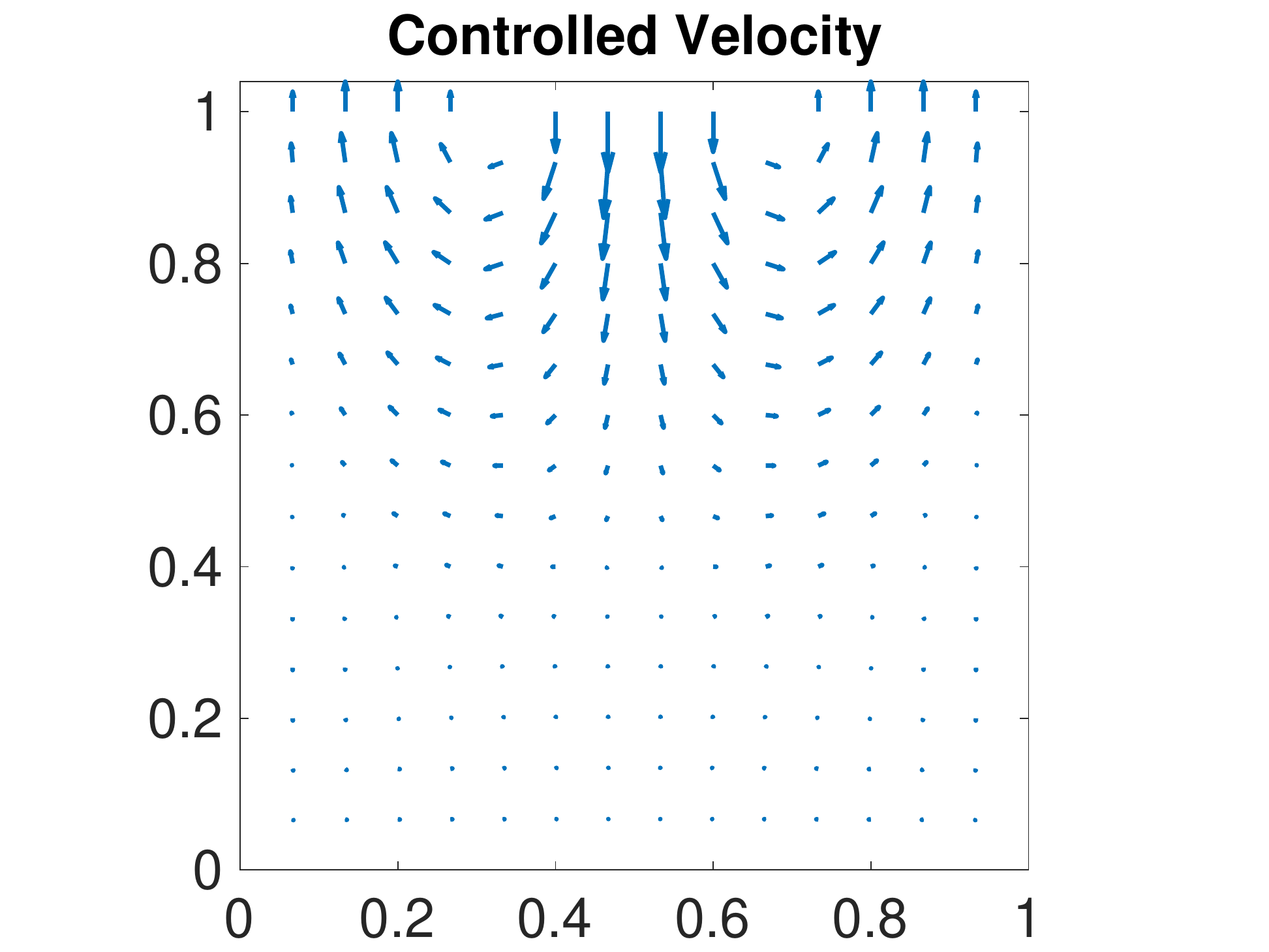}
  \caption{Uncontrolled (left) and controlled (right) velocity field.}
  \label{fig:velocity_field}
\end{figure}

Local sensitivities are evaluated at $N=20$ samples from parameter
space. Figure~\ref{fig:control_solutions} displays the optimal control
solutions for these 20 samples. The left and right panels display the
controller on the left and right boundary, respectively. Each curve
corresponds to the control solution for a different parameter
sample. There is significant variability in the solutions which
indicates a strong dependence of the controller on the uncertain
parameters. Our objective in the hyper-differential sensitivity
analysis is to determine which parameters cause the greatest changes
in the controller so that uncertainty quantification and robust
optimization may focus on them rather than the full set of 153
parameters.

\begin{figure}
\centering
  \includegraphics[width=0.4\textwidth]{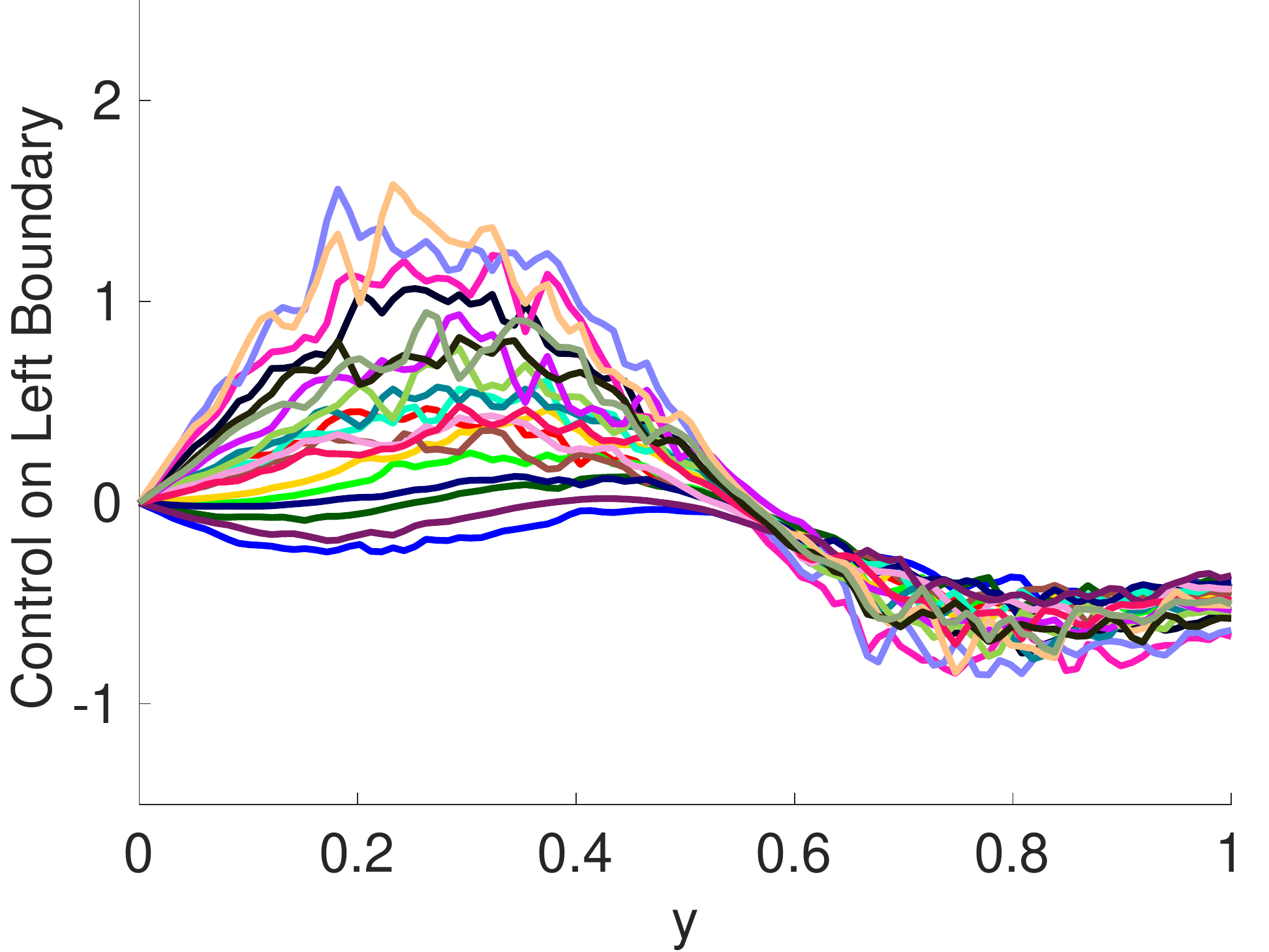}
  \includegraphics[width=0.4\textwidth]{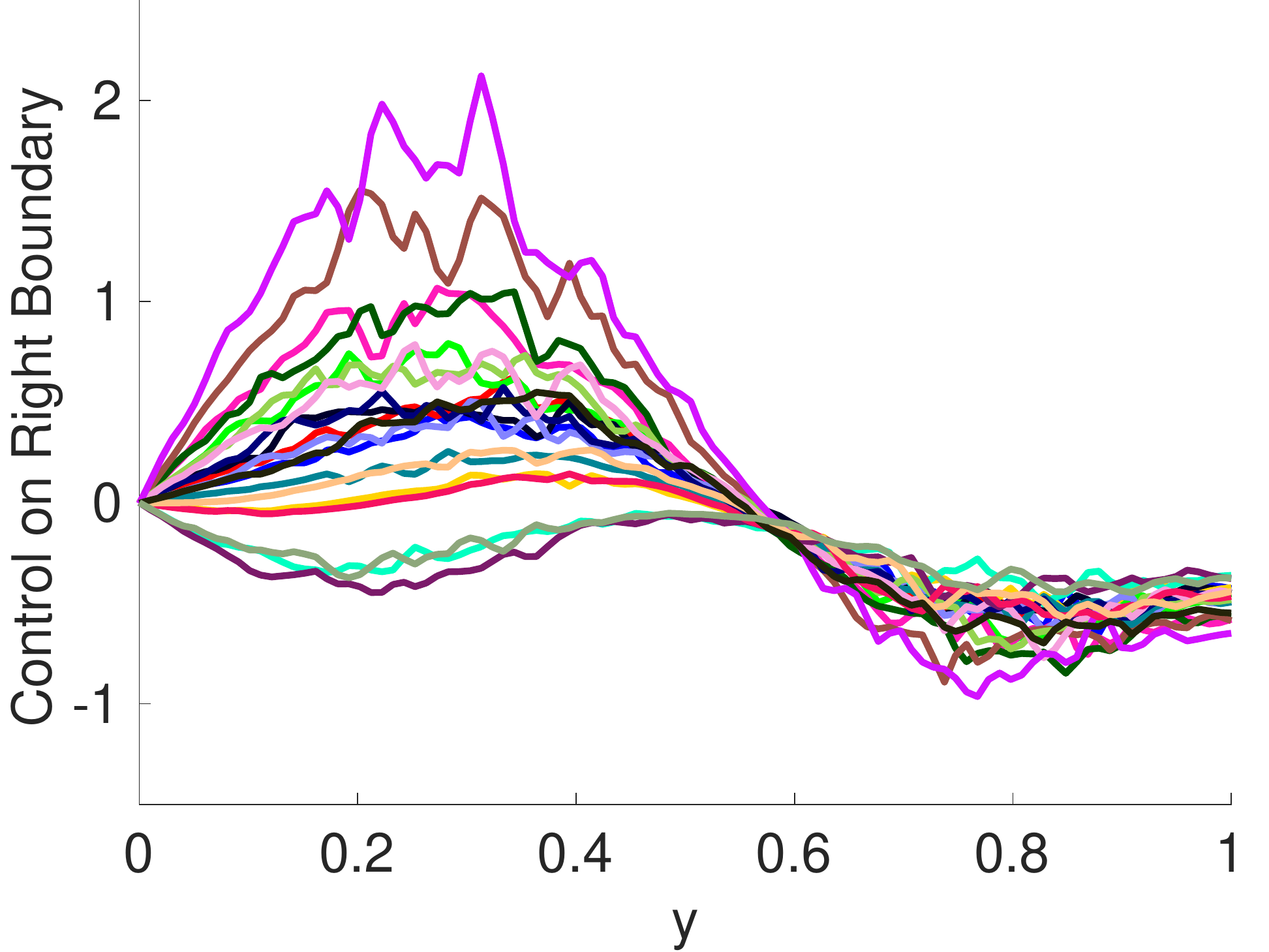}
  \caption{Control solutions for \eqref{opt_boussinesq} corresponding
    to 20 different parameter samples (each parameter from an independent uniform distribution on $[-1,1]$). The left and right panels are
    the controllers on the left and right boundaries,
    respectively. Each curve is a control solution for a given
    parameter sample.}
  \label{fig:control_solutions}
\end{figure}

We compute the leading $K=4$ singular triples of \eqref{sen_operator}
and follow the approach presented in Section~\ref{sec:postprocessing}
to analyze them. An oversampling factor of $L=8$ is
used. Figure~\ref{fig:singular_values} shows the leading 4 singular
values from each of the 20 parameter samples. Each vertical slice in
Figure~\ref{fig:singular_values} gives the 4 singular values for the
fixed parameter sample. We observe that there are 2 dominant singular
triples (thus validating that $K=4$ is sufficiently large) and that
the singular values do not vary significantly over the different
parameter samples.

\begin{figure}
\centering
  \includegraphics[width=0.4\textwidth]{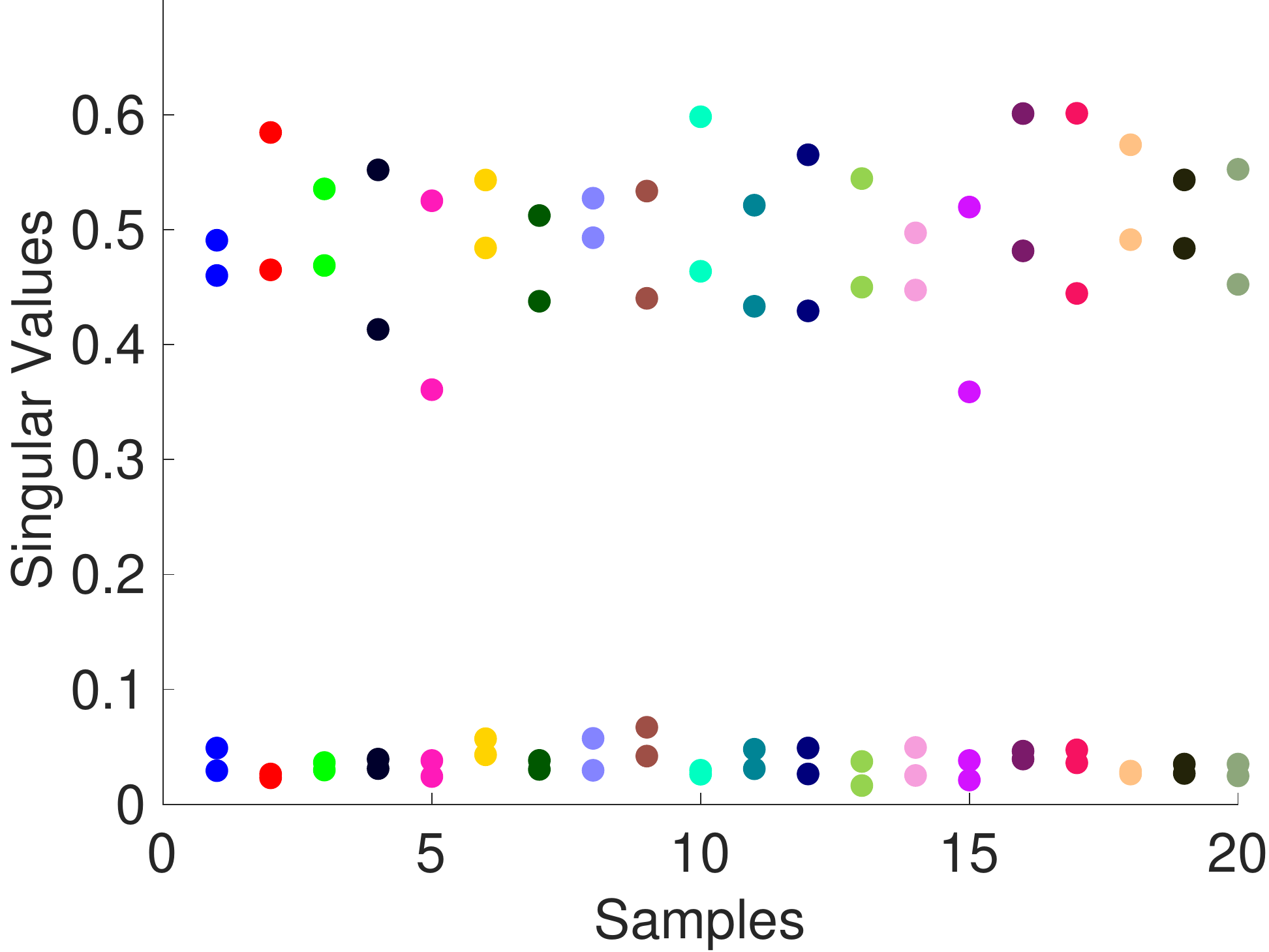}
  \caption{Leading 4 singular values of \eqref{sen_operator} at 20
    different parameter samples. Each vertical slice corresponds to
    the leading 4 singular values for a fixed sample.}
  \label{fig:singular_values}
\end{figure}

The sensitivity indices \eqref{sensitivity_index} and displayed in
Figure~\ref{fig:parameter_sensitivities}. There are 20 circles in each
vertical slice of Figure~\ref{fig:parameter_sensitivities}
corresponding to the local sensitivity index for a fixed parameter
over the 20 samples. We observe several interesting features:
\begin{enumerate}
\item[$\bullet$] The local sensitivity analysis yields similar results for each parameter sample.
\item[$\bullet$] Only around $10\%$ of the uncertain parameters exhibit significant influence on the control strategy.
\item[$\bullet$] The bottom boundary condition, $T_b$, has the greatest influence on the control strategy.
\item[$\bullet$] The cosine components of $T_b$ are more important in the first two frequencies, the sine component is more important in the third frequency.
\end{enumerate}

In the language of Theorem~\ref{sensitivity_opt_thm}, we may postulate that the Lipschitz constant (a measure of nonlinearity for the parameter to optimal solution mapping) $Q$ is ``small" in this problem because the local sensitivities did not vary significantly, albeit, we cannot compute $Q$. Coupling Theorem~\ref{sensitivity_opt_thm} and Figure~\ref{fig:parameter_sensitivities} provides some level of confidence that averaging the 20 local sensitivity indices gives a reasonable estimate of the global sensitivity indices and the statistical characteristics of the optimal solution as the parameters vary.

\begin{figure}
\centering
  \includegraphics[width=0.4\textwidth]{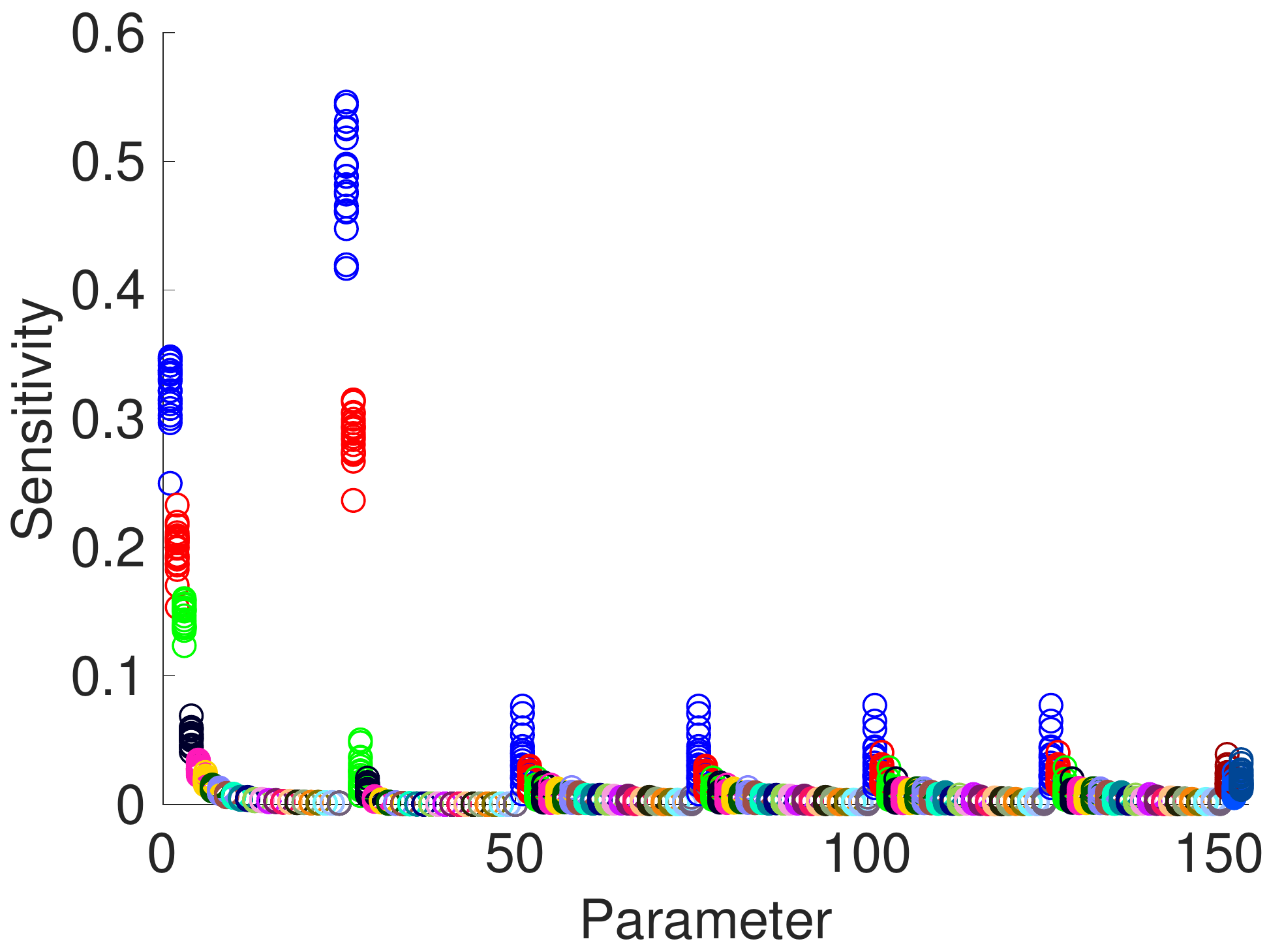}
  \caption{Local sensitivities \eqref{sensitivity_index} for the 153
    uncertain parameters in \eqref{opt_boussinesq}. The 20 circles in
    each vertical slice indicates the sensitivity index for a fixed
    parameter as it varies over the 20 parameter samples. The
    repeating color scheme indicates the grouping of parameters as
    they correspond to sine and cosine components of each boundary
    condition.}
  \label{fig:parameter_sensitivities}
\end{figure}

\begin{figure}[h]
\centering
  \includegraphics[width=0.4\textwidth]{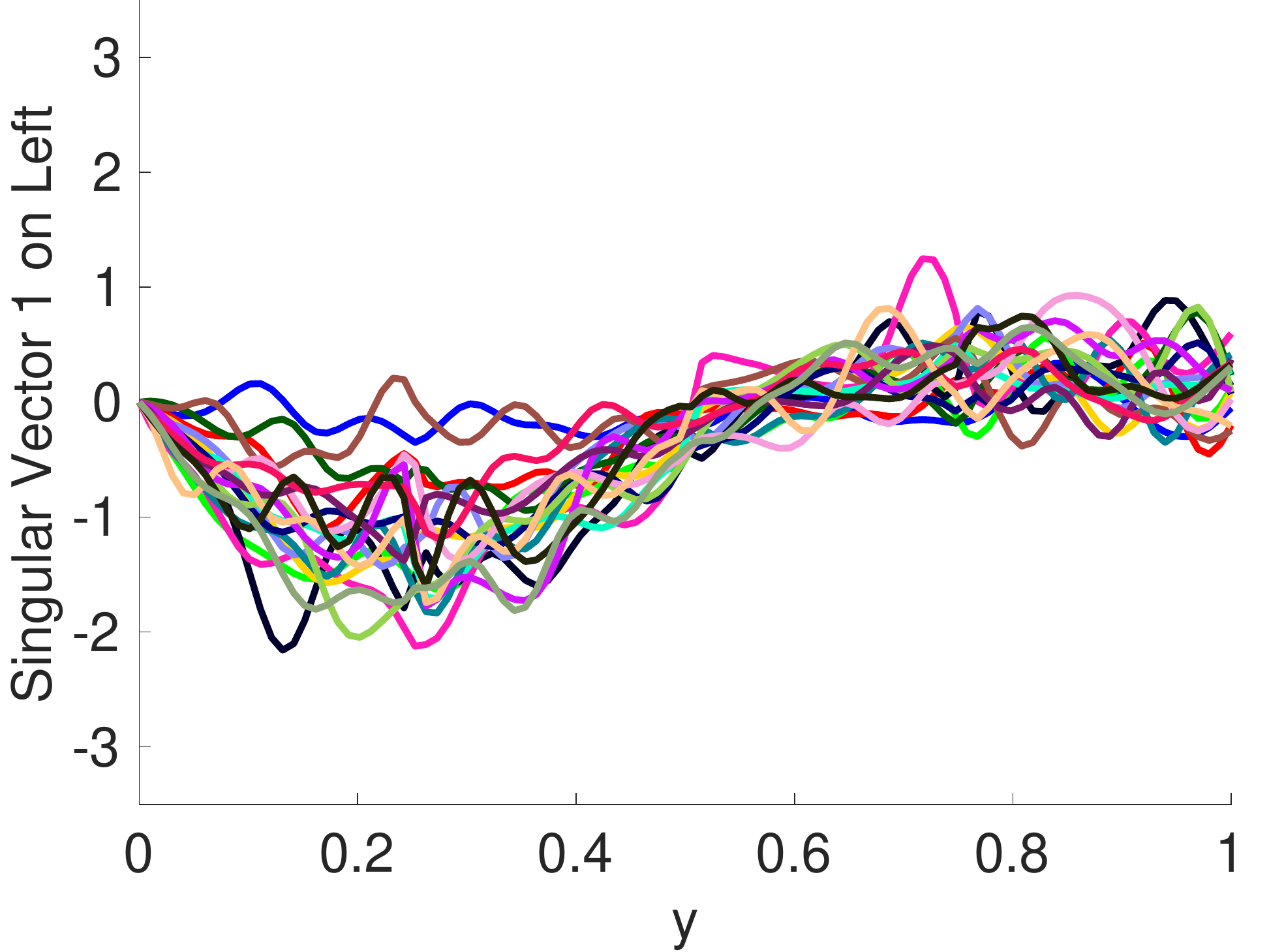}
      \includegraphics[width=0.4\textwidth]{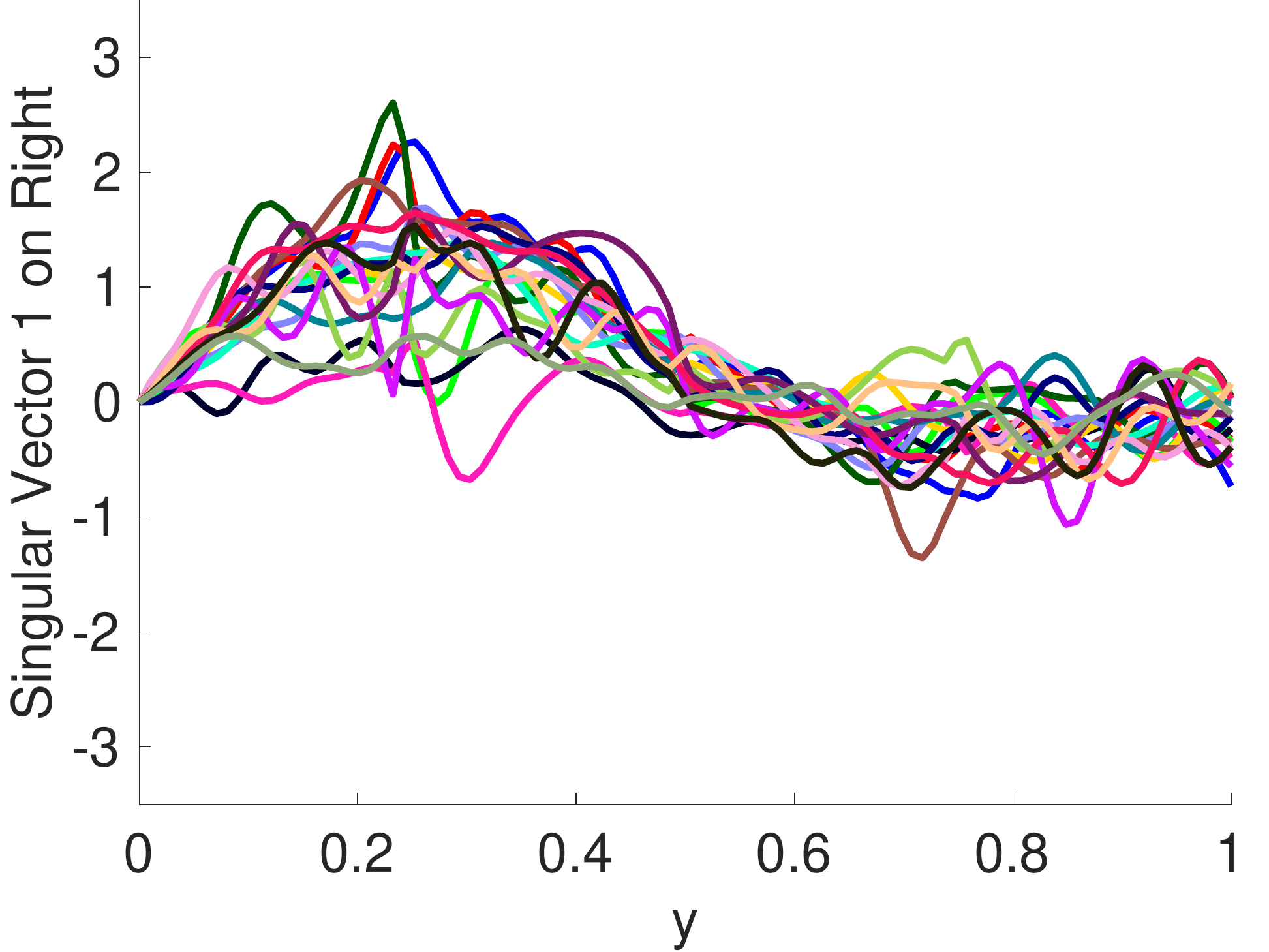}\\
  \includegraphics[width=0.4\textwidth]{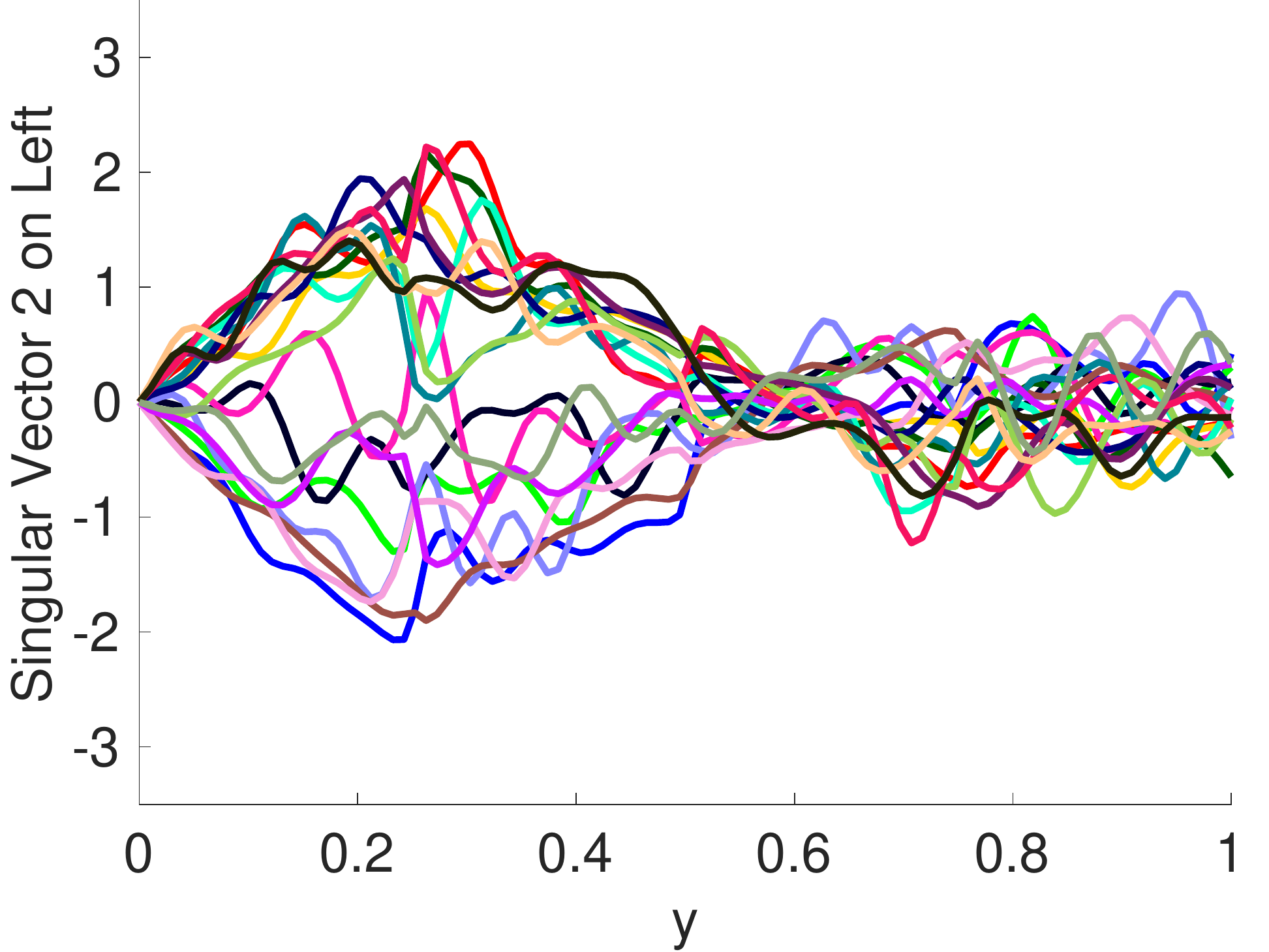} 
  \includegraphics[width=0.4\textwidth]{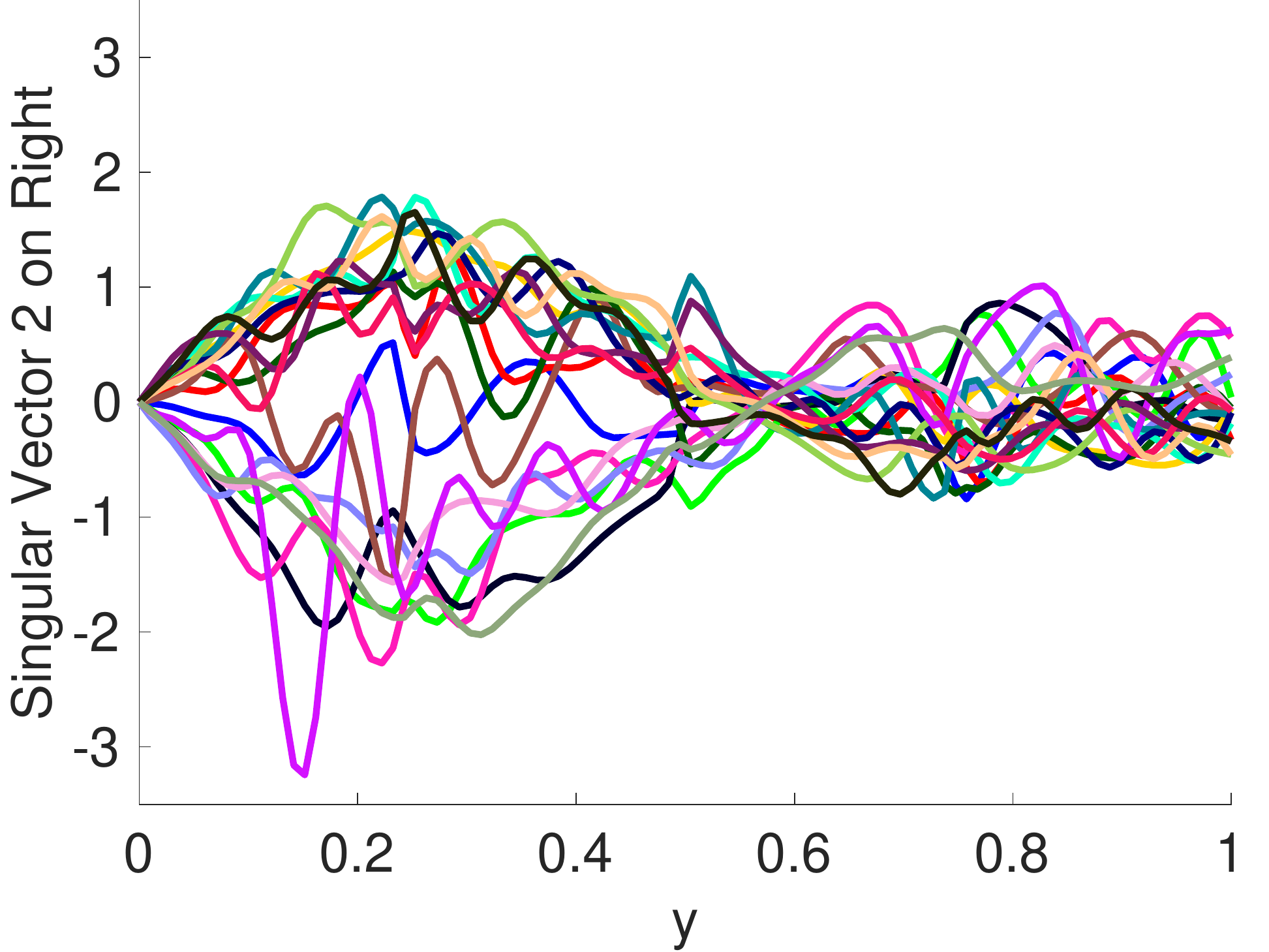} \\
  \caption{First two singular vectors $z_k$, $k=1,2,$ of \eqref{sen_operator}, i.e. $\D \theta_k = \sigma_k z_k$. The top (bottom) row shows the first (second) singular vector on the left and right boundaries, respectively. Each curve corresponds to a different parameter sample.}
  \label{fig:control_sensitivities}
\end{figure}

To complement these parameter sensitivities,
Figure~\ref{fig:control_sensitivities} displays the singular vectors,
$z_k$, $k=1,2,$ corresponding to the leading parameter perturbations
$\theta_k$, $k=1,2$. The top left and top right panels of
Figure~\ref{fig:control_sensitivities} show the first singular vector
on the left and right boundaries, respectively; the bottom left and
bottom right panels are the second singular vector on the left and
right boundaries, respectively. The singular vectors $z_1$ and $z_2$
may be interpreted as the change in the control strategy if the
parameters are perturbed according to the singular vectors $\theta_1$
and $\theta_2$, respectively. We observe that the control strategy
will change more near the bottom of the domain, an unsurprising result
since the greatest sensitivity is in the bottom boundary condition.

The computation was performed using 80 compute nodes, each containing
16 processors. By taking 20 parameter samples and 16 random vectors in
the eigenvalue solver, the embarrassingly parallel loop in the
eigenvalue solver distributed the computation over all 1280
processors, using 4 processors in parallel for each matrix vector
product. This reduces the overall execution time to approximately one
matrix vector product in Line 5 and one matrix vector product in Line
9 of Algorithm~\ref{alg:hyperdiff}, each leveraging parallel linear
algebra with 4 processors. Since the KKT solve \eqref{KKT_solve}
dominates the computational cost, the total execution time for
computing 20 local sensitivities is approximately equal to 4 KKT
solves.

\subsection{Subsurface Contaminant Source Inversion}
In this subsection we consider an inverse problem constrained by the
transient advection diffusion equation and Darcy's equation in two
spatial dimensions. This emulates a model for subsurface contaminant
transport. We seek to invert for a contaminant source given sparse
noisy measurements of the contaminant. Consider the following inverse
formulation:
\begin{align} 
& \min\limits_{p,c,z} \frac{1}{2} \sum\limits_{i=1}^T \sum\limits_{j=1}^M (\mathcal P_{i,j} c-d_{i,j})^2 + \frac{\alpha}{2} \int_{\Omega} z^2 dx \label{source_inv} \\
& s.t. \nonumber \\
& \nabla \cdot (-\kappa(\theta) \nabla p) = 0 & \text{ in } \Omega \nonumber \\
& \frac{\partial c}{\partial t} + \nabla \cdot (-\epsilon(\theta) \nabla c) + (-\kappa(\theta) \nabla p) \cdot \nabla c = \chi_{[.01,.02]}(t)z & \text{ in } \Omega \text{ for } t>0  \nonumber \\
& p=\psi(\theta) & \text{on } \Gamma_L \cup \Gamma_R  \nonumber \\
& -\kappa(\theta) \nabla p \cdot n =0 & \text{on } \Gamma_B \cup \Gamma_T \nonumber \\
& \nabla c \cdot n =0 & \text{on } \Gamma \text{ for } t>0  \nonumber \\
& u =0 & \text{in } \Omega \text{ for } t=0
\end{align}
where $\Omega=(0,1)^2$ with boundary $\Gamma=\Gamma_L \cup \Gamma_B
\cup \Gamma_R \cup \Gamma_T$, $\Gamma_L=\{0\} \times (0,1)$,
$\Gamma_B=(0,1) \times \{0\}$, $\Gamma_R=\{1\} \times (0,1)$,
$\Gamma_T=(0,1)\times \{1\}$, and $n$ denotes the outward pointing
normal vector to the boundary.

The state consists of the transient contaminant concentration $c$ and
the steady state pressure $p$; the stationary source $z$ is a function
defined on $\Omega$ which appears on the right hand side of the
advection diffusion equation along with the characteristic function
$\chi_{[.01,.02]}(t)$ which equals $1$ on the time interval
$[.01,.02]$ and $0$ otherwise. Contaminant concentration data is
collected at $M=121$ sensors which are uniformly distributed on an
$11\times11$ grid in $[0,1]^2$ and $T=40$ instances in time. Synthetic
data $d_{i,j}$ (at the $i^{th}$ time instance and $j^{th}$ sensor
location) is generated by solving the governing PDEs (with double the
mesh resolution used in the inverse problem) with a specified source
term (given in Figure~\ref{fig:source}) and adding independent
Gaussian noise whose standard deviation is $3\%$ of the observed
concentration. The observation operator $\mathcal P_{i,j}$ maps the
PDE solution to the contaminant concentration at the $i^{th}$ time
instance and $j^{th}$ sensor location.

Uncertainties enter the the model through the finite dimensional
vector (coming from a spatial discretization) $\theta \in \R^m$ where
$m=256+1+16+16=289$ which parameterizes perturbations of the scalar
and spatially dependent parameters in the PDE, in particular, the
permeability field $\kappa$, left and right Dirichlet boundary
condition $\psi$, and diffusion coefficient $\epsilon$. These
parameters are fixed to nominal values in order to solve the inverse
problem and we consider the sensitivity of the source estimate to
perturbations of the parameters.

The nominal value of the permeability field, denote it by
$\overline{\kappa}$, is shown in Figure~\ref{fig:pressure} alongside
the Darcy pressure generated by solving Darcy's equation with the
nominal permeability field $\overline{\kappa}$ and nominal boundary
condition defined by
\begin{align*}
\overline{\psi}(x,y) =
 \left\{\begin{array}{ll}  10+2\cos(2 \pi y) & \mbox{ if } x \in \Gamma_L \vspace{.2 cm} \\
   12 + \cos(2 \pi y) + .5 \cos(4 \pi y)  & \mbox{ if } x \in \Gamma_R
  \end{array} \right. .
\end{align*}
The nominal scalar diffusion coefficient in the advection diffusion
equation is $\overline{\epsilon} = 1$.

To study the sensitivity of the source estimate to uncertainties in
the permeability field, boundary condition, and diffusion coefficient,
we represent them as
\begin{align*}
\kappa &= \overline{\kappa}\left(1 + a \sum\limits_{k=1}^{(L+1)^2} \theta_k \phi_k \right) \\
\psi &= \overline{\psi}\left(1 + a \sum\limits_{k=1}^{2(L+1)} \theta_{(L+1)^2+k} \eta_k\right)\\
\epsilon &= \overline{\epsilon}\left(1 + a \theta_{(L+1)^2 + 2(L+1) + 1}\right)
\end{align*}
where $a=0.2$ represents the level of uncertainty ($20\%$), $L=15$ is
an integer specifying a spatial discretization, and $\phi_k$ and
$\eta_k$ are linear finite element basis functions defined on a
rectangular mesh with $L+1$ equally spaced nodes in each spatial
dimension ($\phi_k$ is defined on $[0,1]^2$ and $\eta_k$ is defined on
$[0,1]$). The parameter weighting matrix $M_p$ is block diagonal with
its blocks given by the finite element mass matrices arising from
discretization of $\kappa$ and $\psi$, and the scalar identity for the
$1\times 1$ block corresponding to $\epsilon$.

The advection-diffusion equation is discretized on the time mesh
$t_i=\frac{.11i}{44}$, $i=0,2,\dots,44$, and solved using backward
Euler time stepping. Both the pressure and contaminant concentration
are spatially discretized with a 2,601 degrees of freedom linear
finite element approximation on a rectangular grid. The spatially
dependent parameters are discretized on a coarser mesh (256 degrees of
freedom) to enforce smoothness of the perturbations. The data
$d_{i,j}$ is generated by solving on a mesh with 10,201 degrees of
freedom. A regularization coefficient $\alpha=0.0005$ is used and the
optimization problem is solved using a truncated CG trust region
solver in ROL.

\begin{figure}
    \centering
    \includegraphics[width=0.4\textwidth]{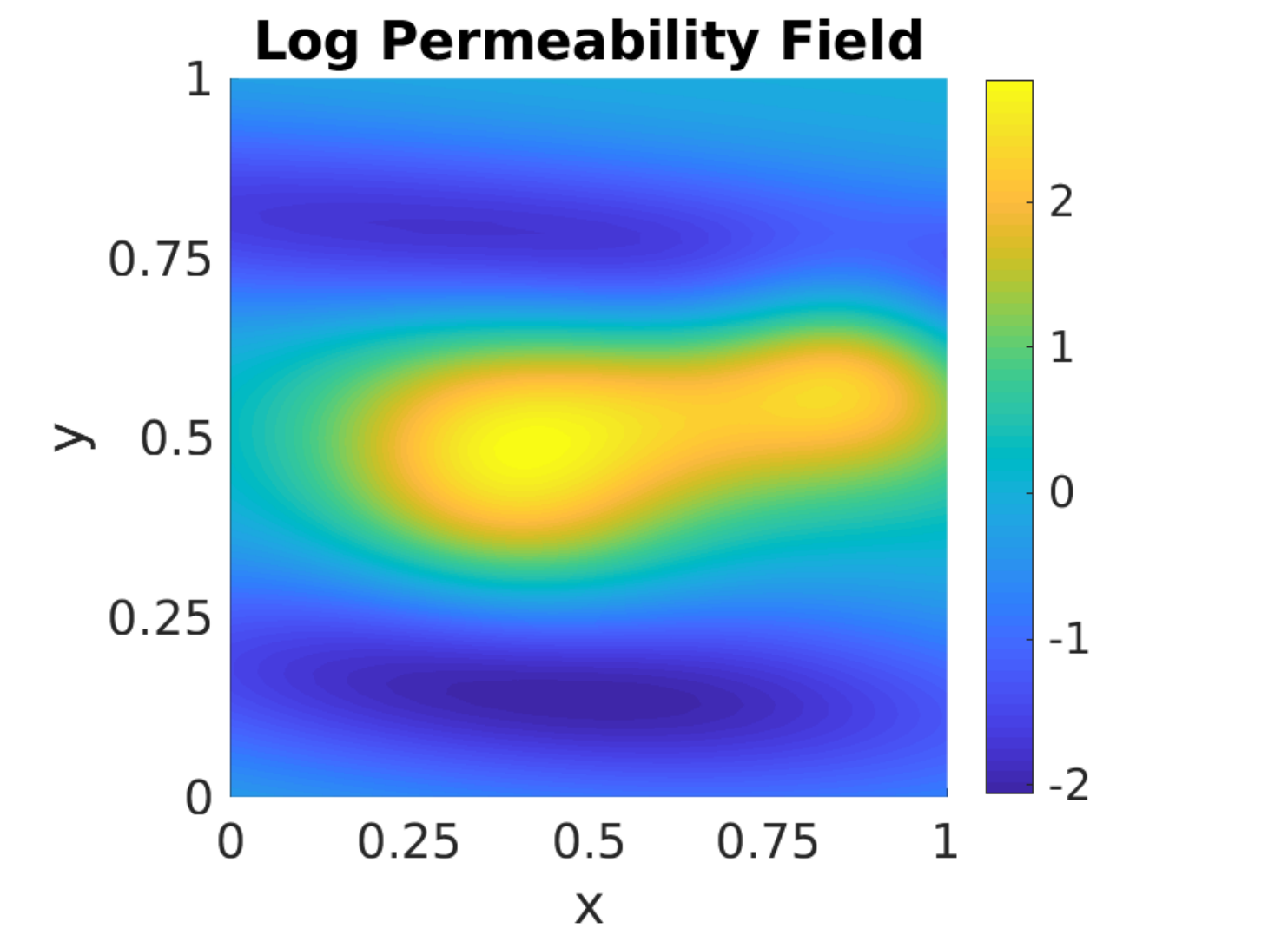}
    \includegraphics[width=0.4\textwidth]{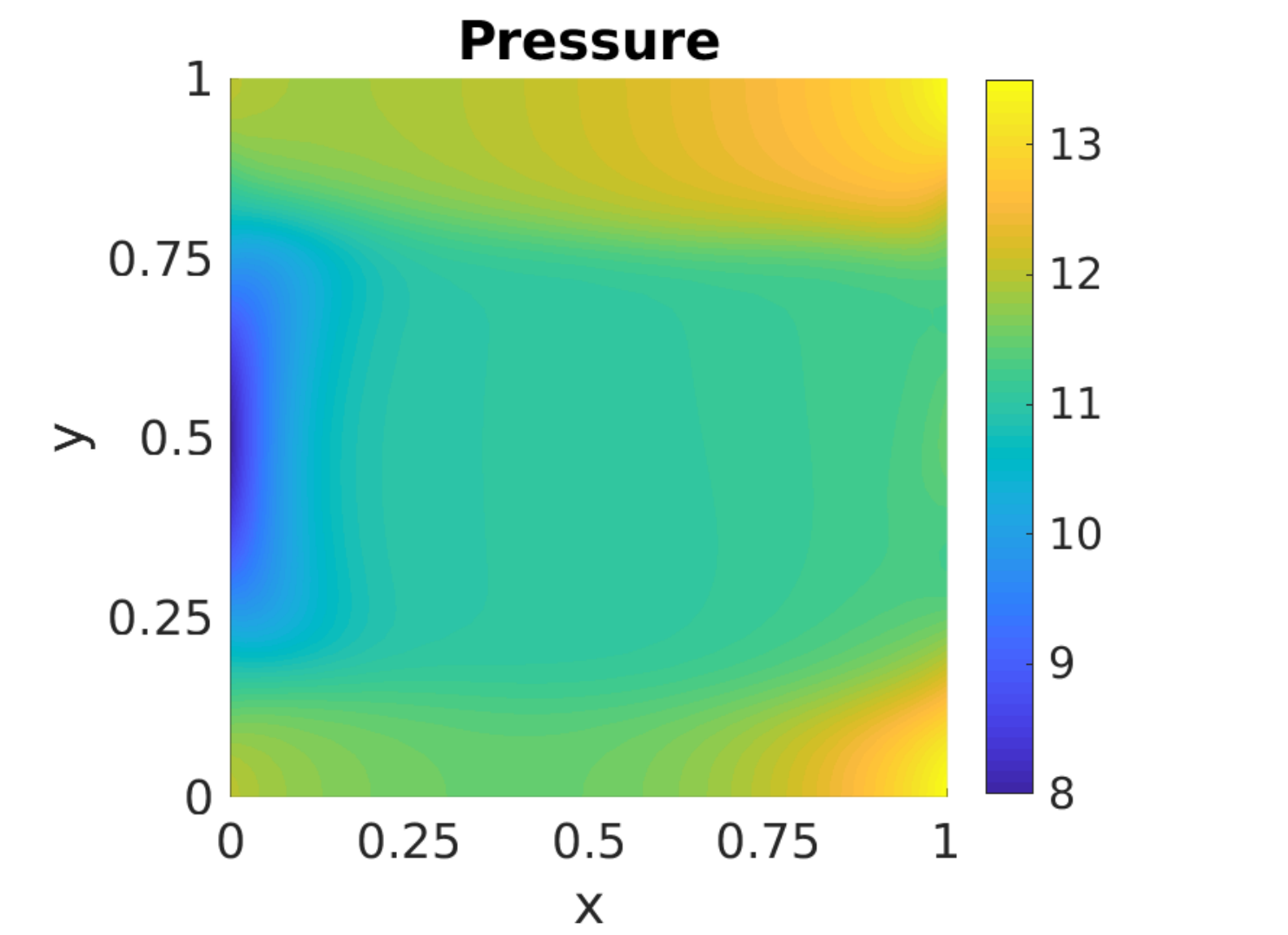}
    \caption{Left: the log (base 10) of the nominal permeability field
      $\overline{\kappa}$; right: the solution of the pressure
      equation with nominal parameters $\overline{\kappa}$ and
      $\overline{\psi}$.}
    \label{fig:pressure}
\end{figure}

The contaminant source used to generate synthetic data is given in the
left panel of Figure~\ref{fig:source} and the source estimated by
solving the inverse problem is given in the right panel of
Figure~\ref{fig:source}. The location and magnitude of the source is
recovered well, albeit, the estimate is noisy as a result of the
sparse noisy data and uninformed regularization.

\begin{figure}
    \centering
    \includegraphics[width=0.4\textwidth]{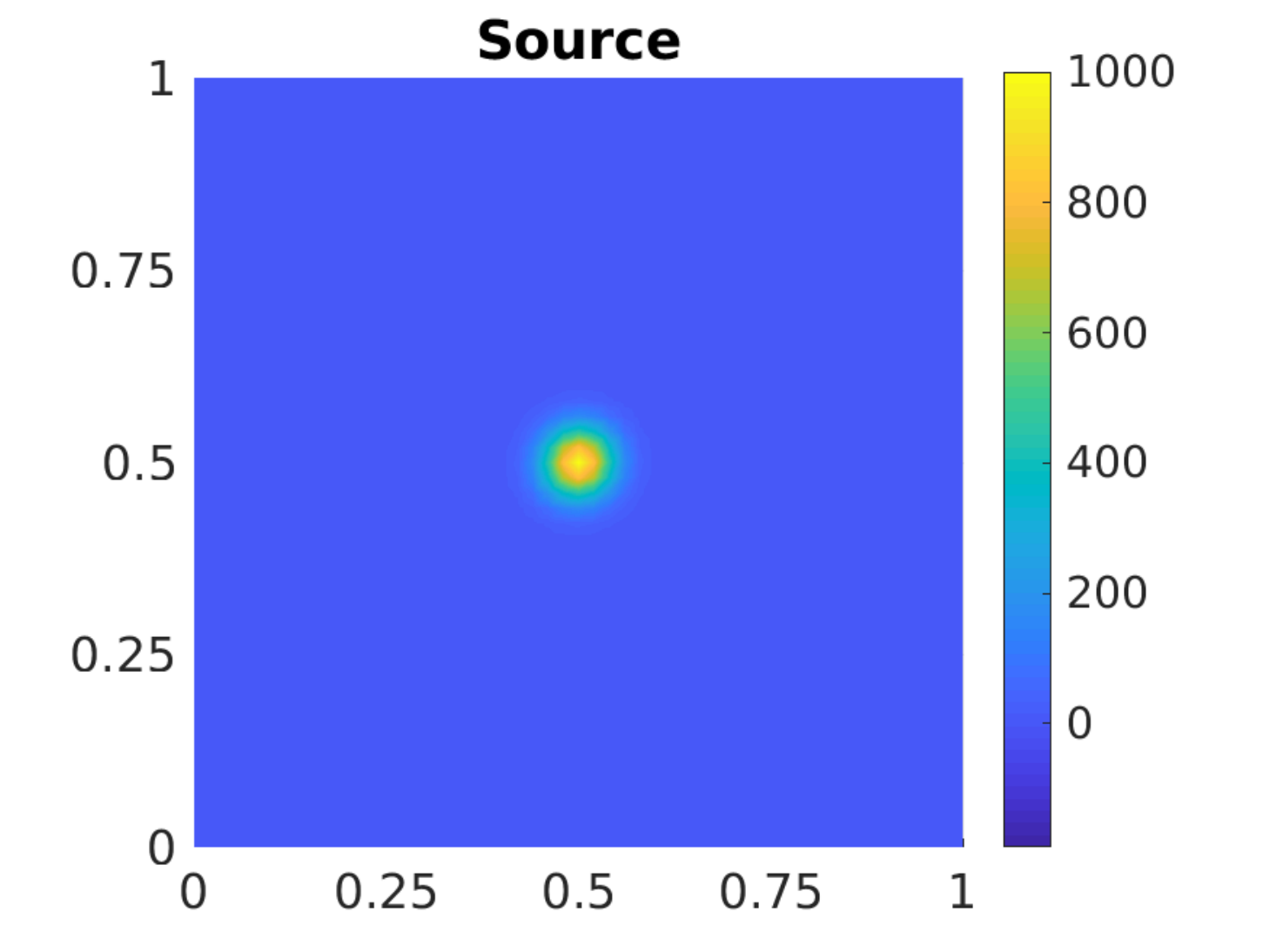}
    \includegraphics[width=0.4\textwidth]{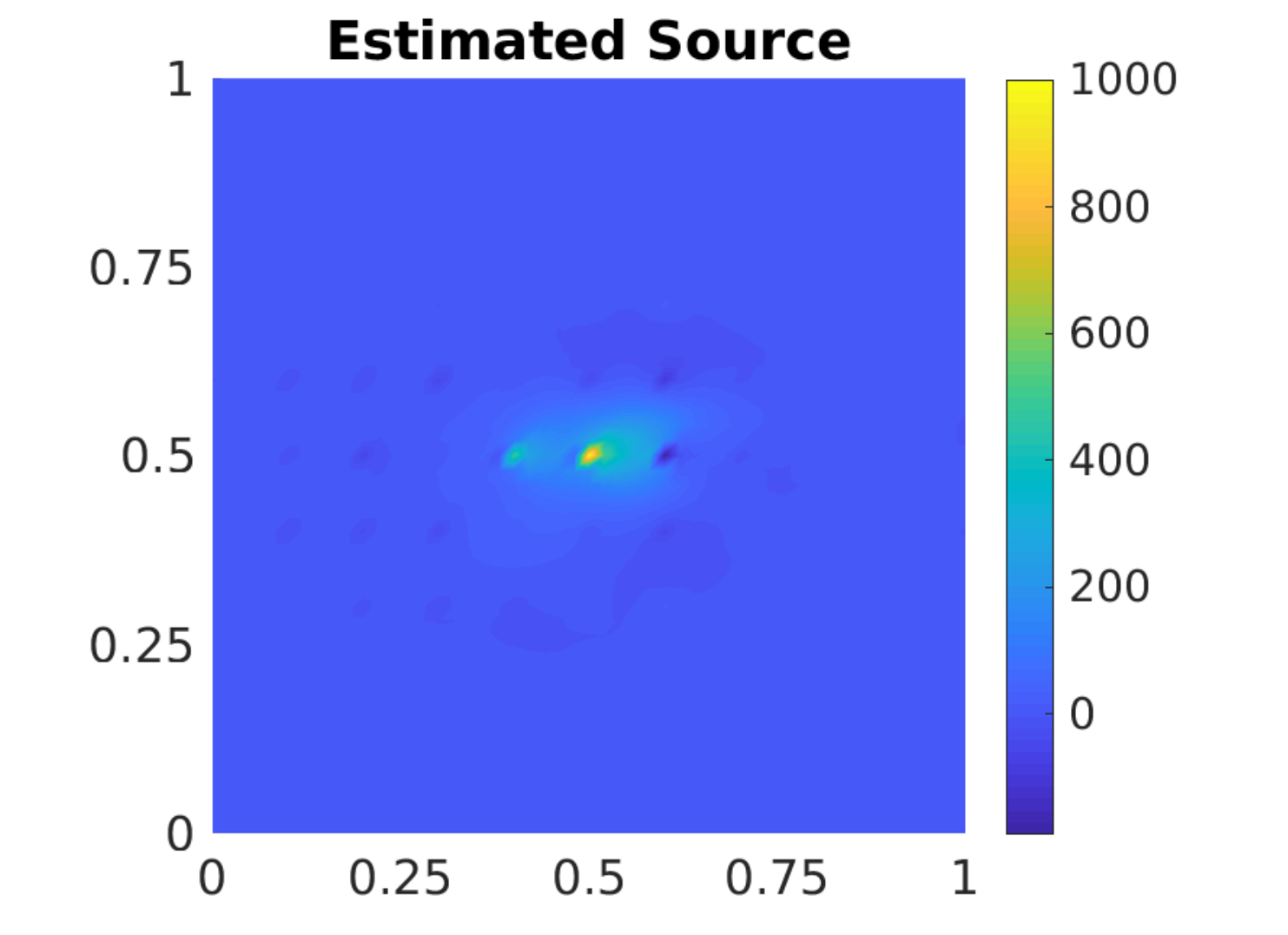}
    \caption{Left: the ``true" source term used to generate synthetic data; right: the estimated source term computed by solving \eqref{source_inv} with sparse noisy data.}
    \label{fig:source}
\end{figure}

The time evolution of the contaminant concentration is given in
Figure~\ref{fig:contaminant} where three time instances are shown
($t=0.02$, $t=0.06$, and $t=0.10$) from left to right. The observed
noisy data is given on the top row and the estimated contaminant
concentration (using the estimated source) is given on the bottom row.

\begin{figure}
    \centering
    \includegraphics[width=0.32\textwidth]{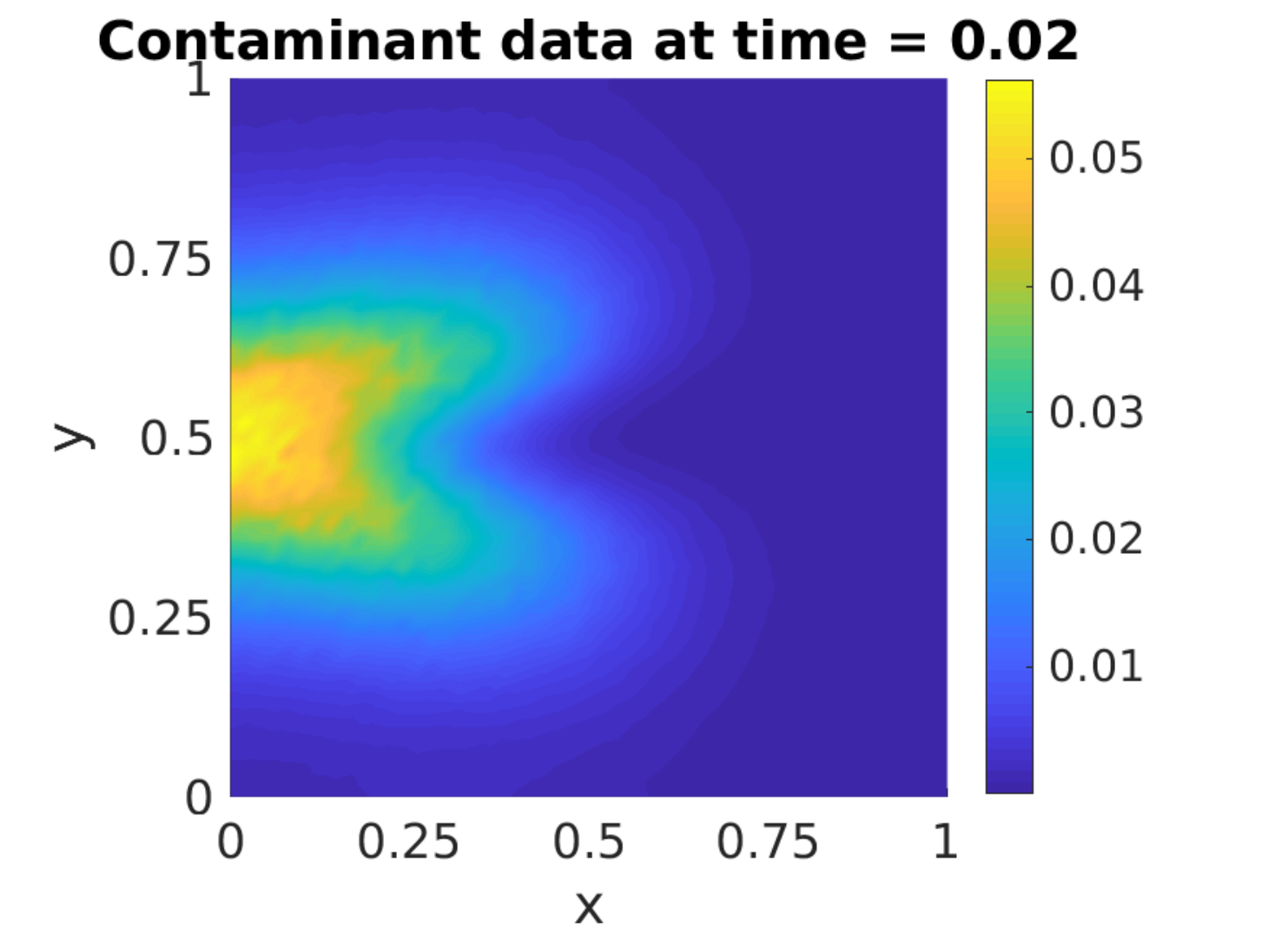}
    \includegraphics[width=0.32\textwidth]{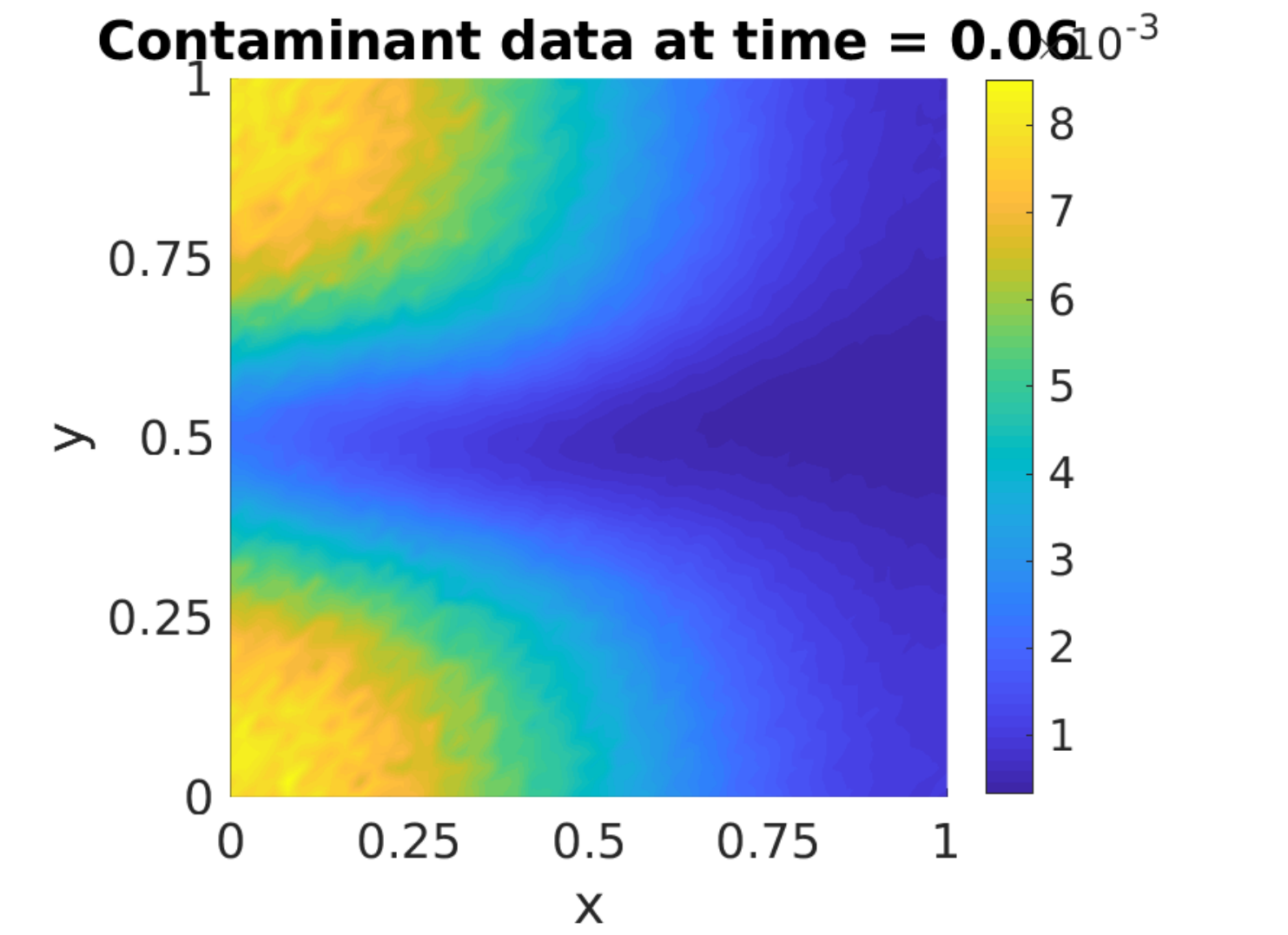}
    \includegraphics[width=0.32\textwidth]{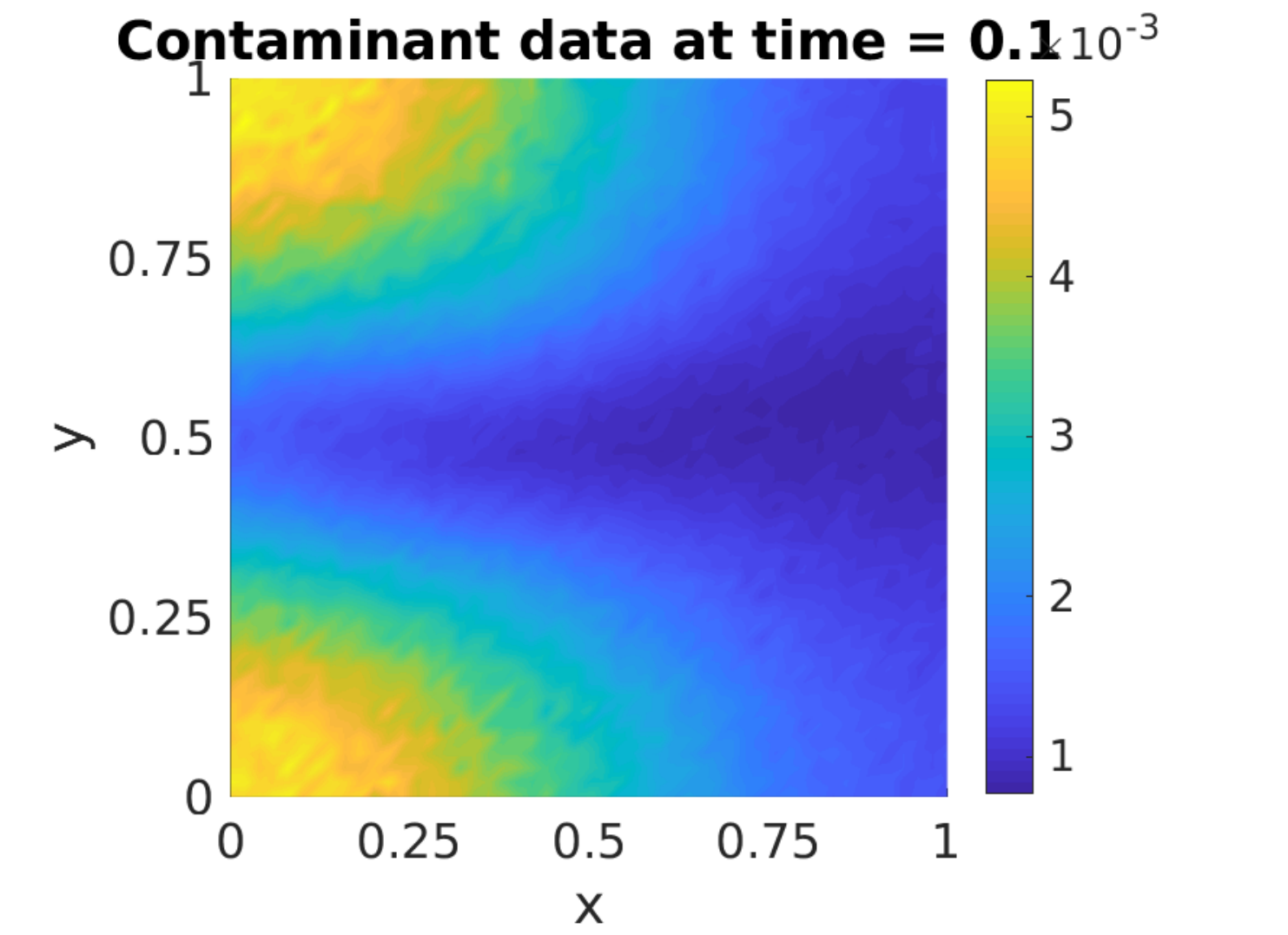}\\
    \includegraphics[width=0.32\textwidth]{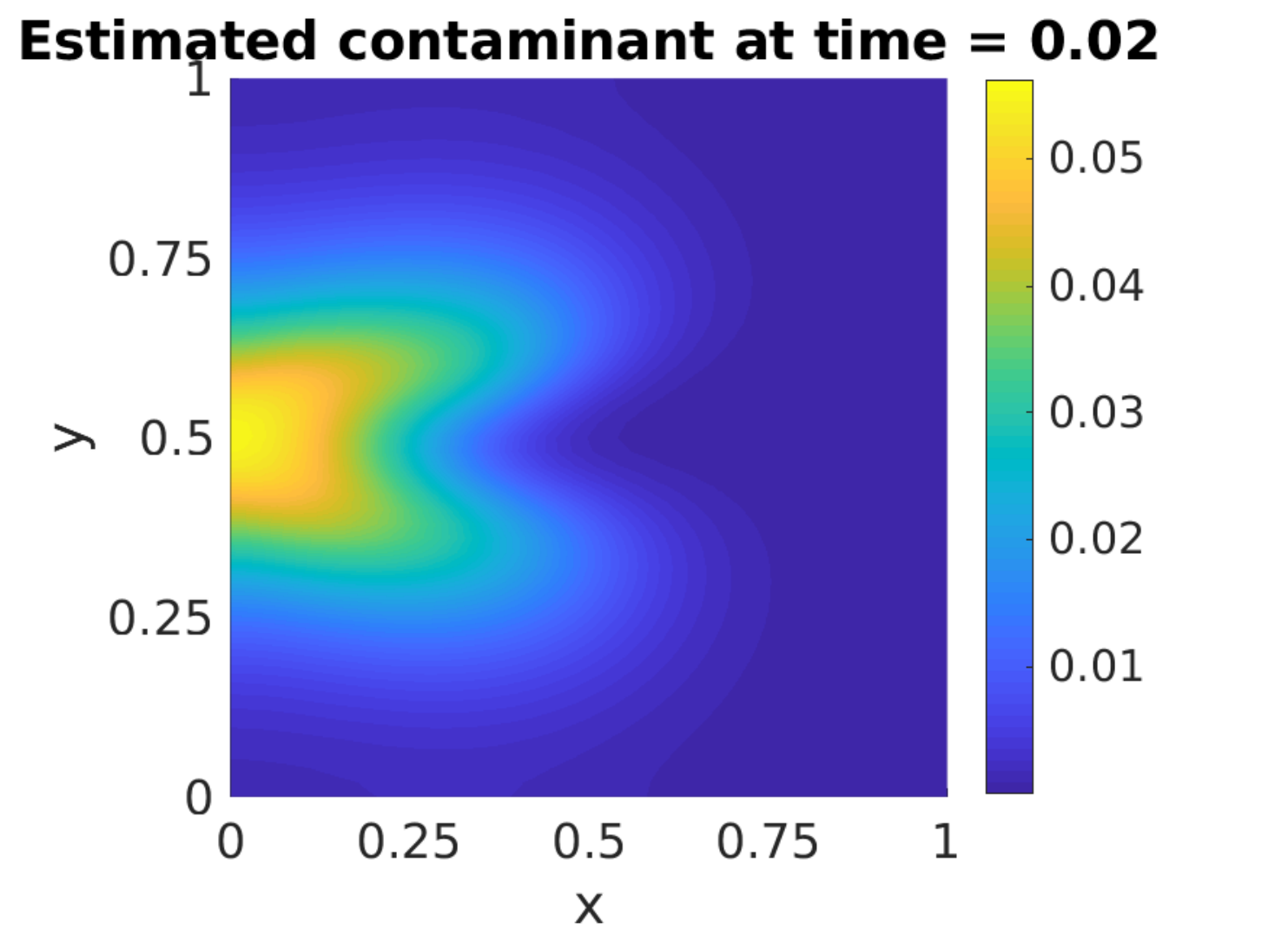}
    \includegraphics[width=0.32\textwidth]{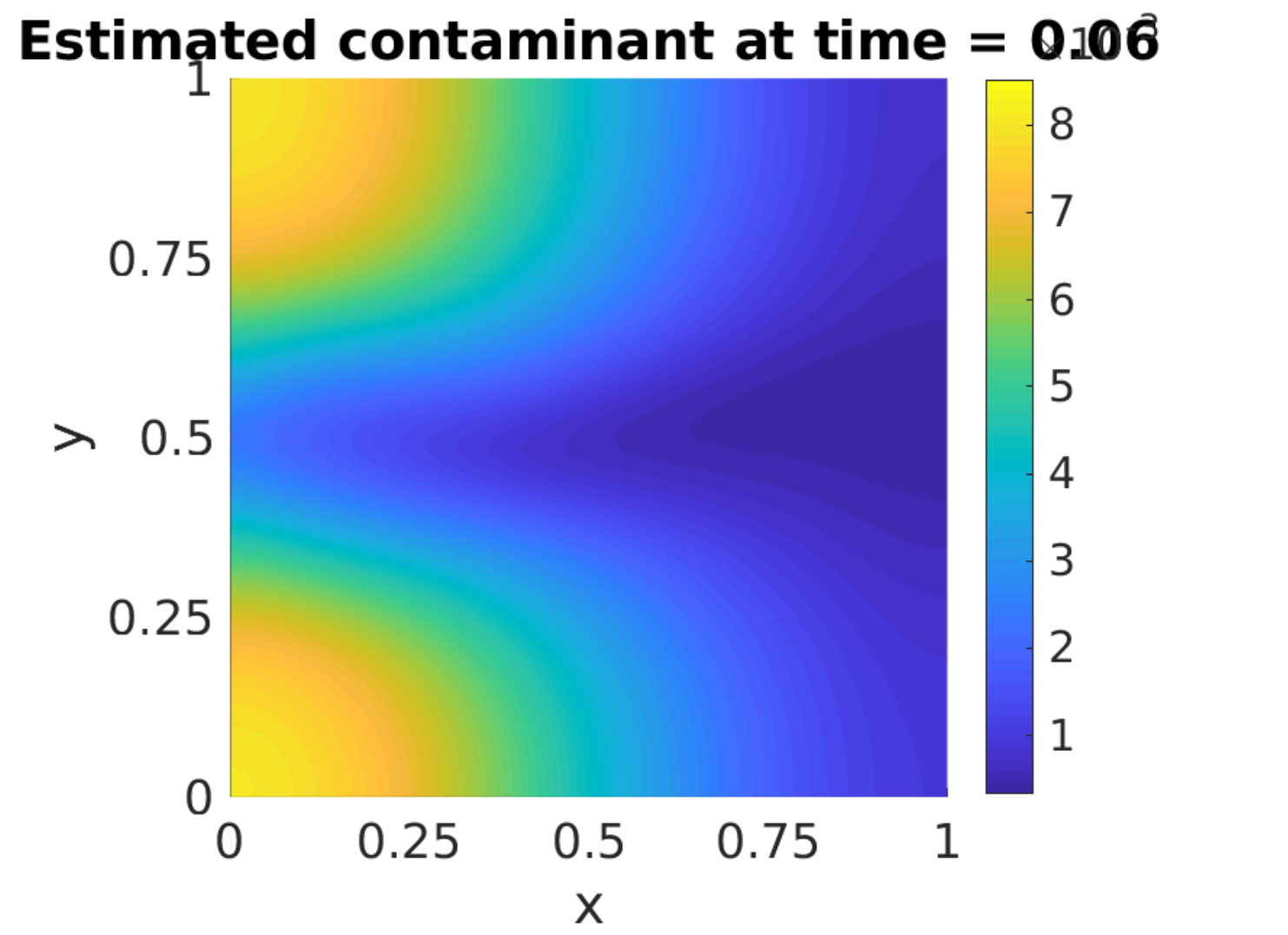}
    \includegraphics[width=0.32\textwidth]{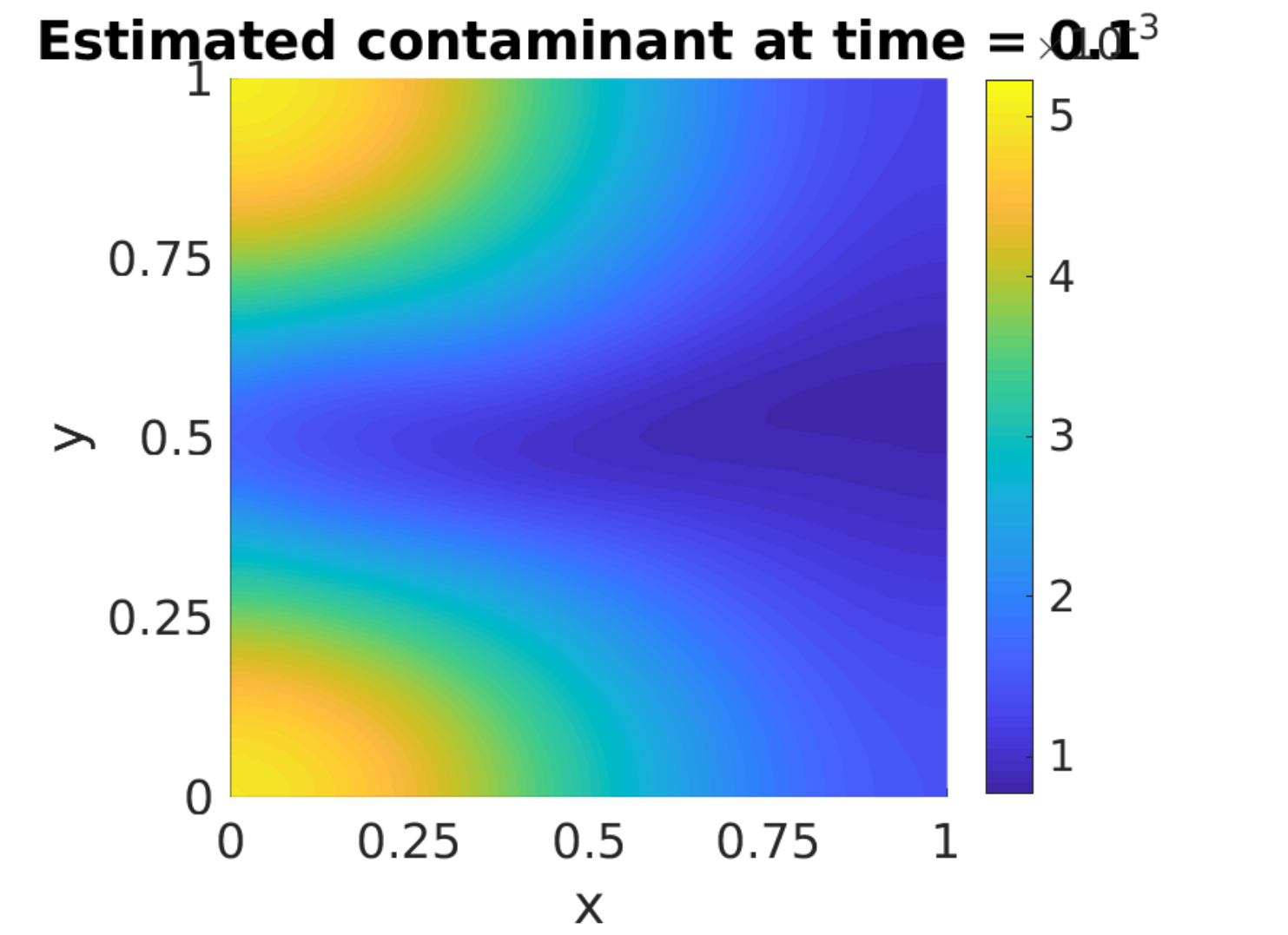}\\
    \caption{Time evolution of the contaminant concentration; from
      left to right, the solutions at times 0.02, 0.06, and 0.10. The
      top row corresponds to the ``true" state polluted by noise; the
      bottom row corresponds to the estimated state using the
      estimated source in the right panel of Figure~\ref{fig:source}.}
    \label{fig:contaminant}
\end{figure}

In this example, we only consider local sensitivity analysis (around
the parameter nominal estimates). In the case of inverse problems, we
interpret the sensitivity indices as the bias created in our source
estimation as a result of potentially misspecifying other parameters
(permeability, boundary conditions, diffusion). The sensitivity
indices indicate where we should invest efforts in characterizing the
parameters in order to ensure a high fidelity source estimation, the
leading singular value provides an error bound on the source
estimation, and the source singular vectors indicate which features of
the estimate will change given the uncertainty in the parameters.

We compute the leading $K=12$ singular triples of \eqref{sen_operator}
with an oversampling factor of $L=8$ is
used. Figure~\ref{fig:source_est_singular_values} shows the leading 12
singular values with an order of magnitude decay from first to last.

\begin{figure}
    \centering
    \includegraphics[width=0.4\textwidth]{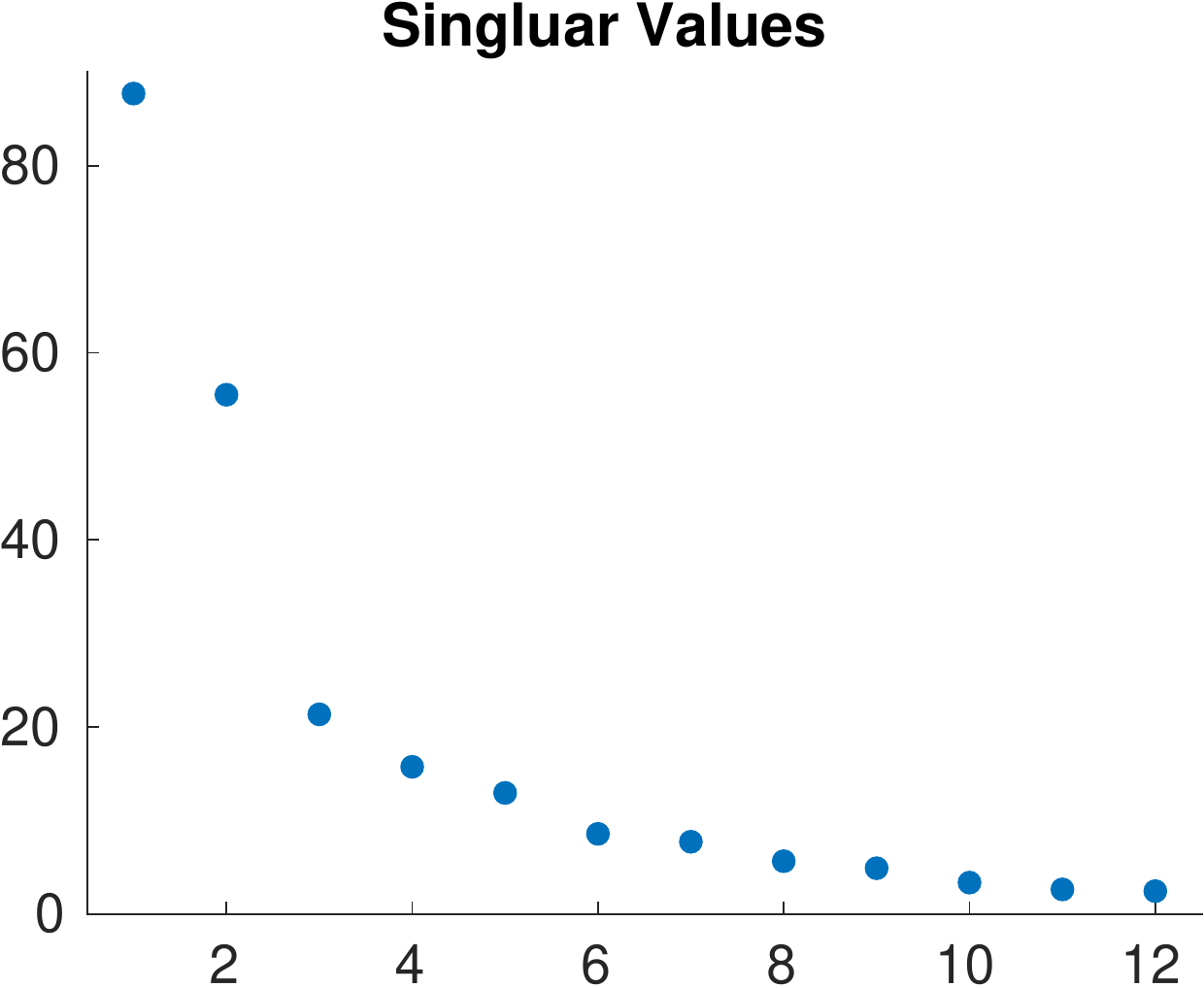}
    \caption{The 12 leading singular values for the source estimation inverse problem sensitivities.}
    \label{fig:source_est_singular_values}
\end{figure}

The sensitivity indices are shown in
Figure~\ref{fig:source_est_sensitivities}. The left, center, and right
panels gives the sensitivities for the parameterization of the
permeability field $\kappa$, boundary condition $\psi$, and diffusion
coefficient $\epsilon$. The center of the basis functions ($\phi$'s
and $\eta$'s) correspond to the spatial location of the sensitivity
indices (dots) in the left and center panels, the diffusion
coefficient is a scalar and hence only one sensitivity index (dot) is
present.
\begin{figure}
    \centering
    \includegraphics[width=0.33\textwidth]{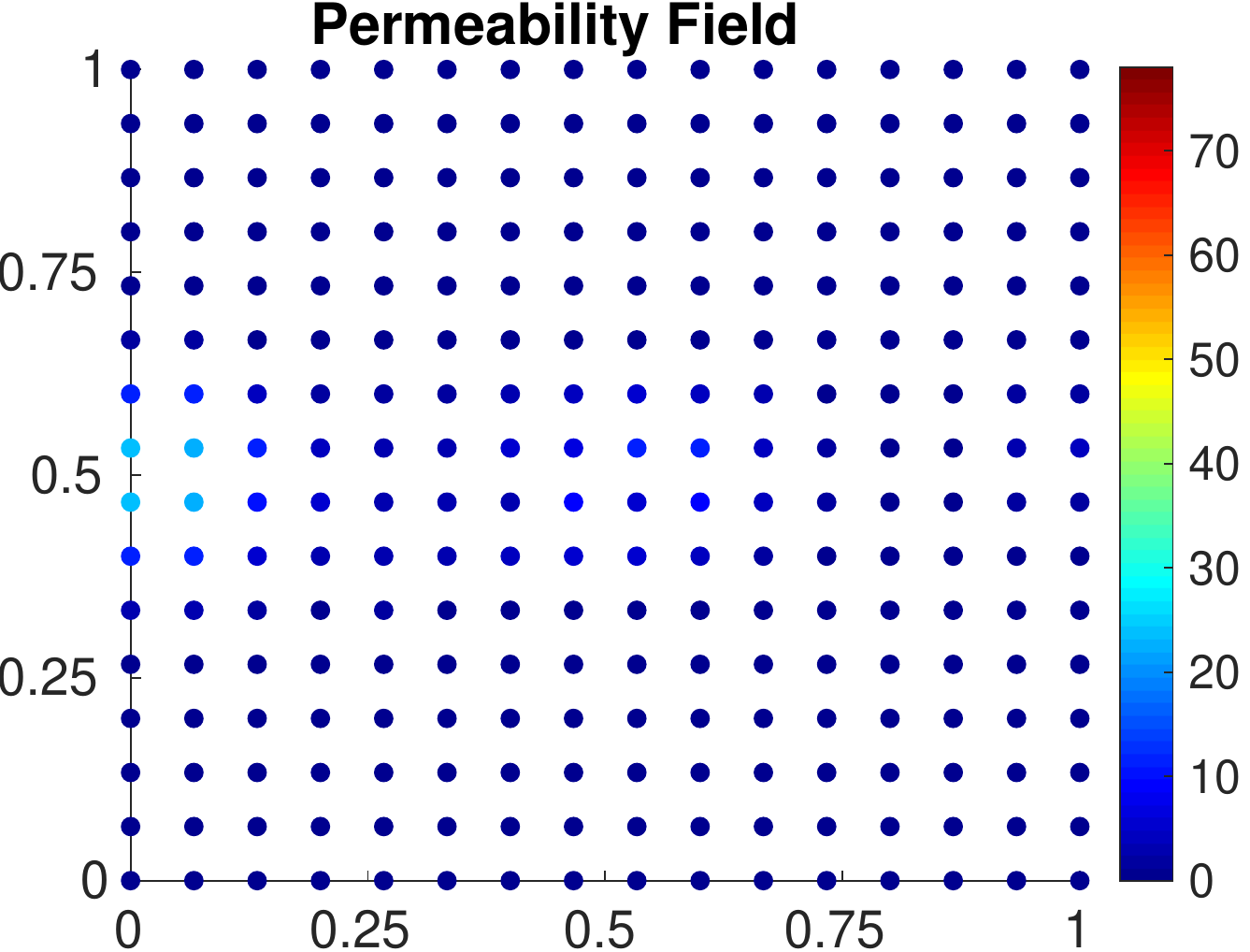}
    \includegraphics[width=0.33\textwidth]{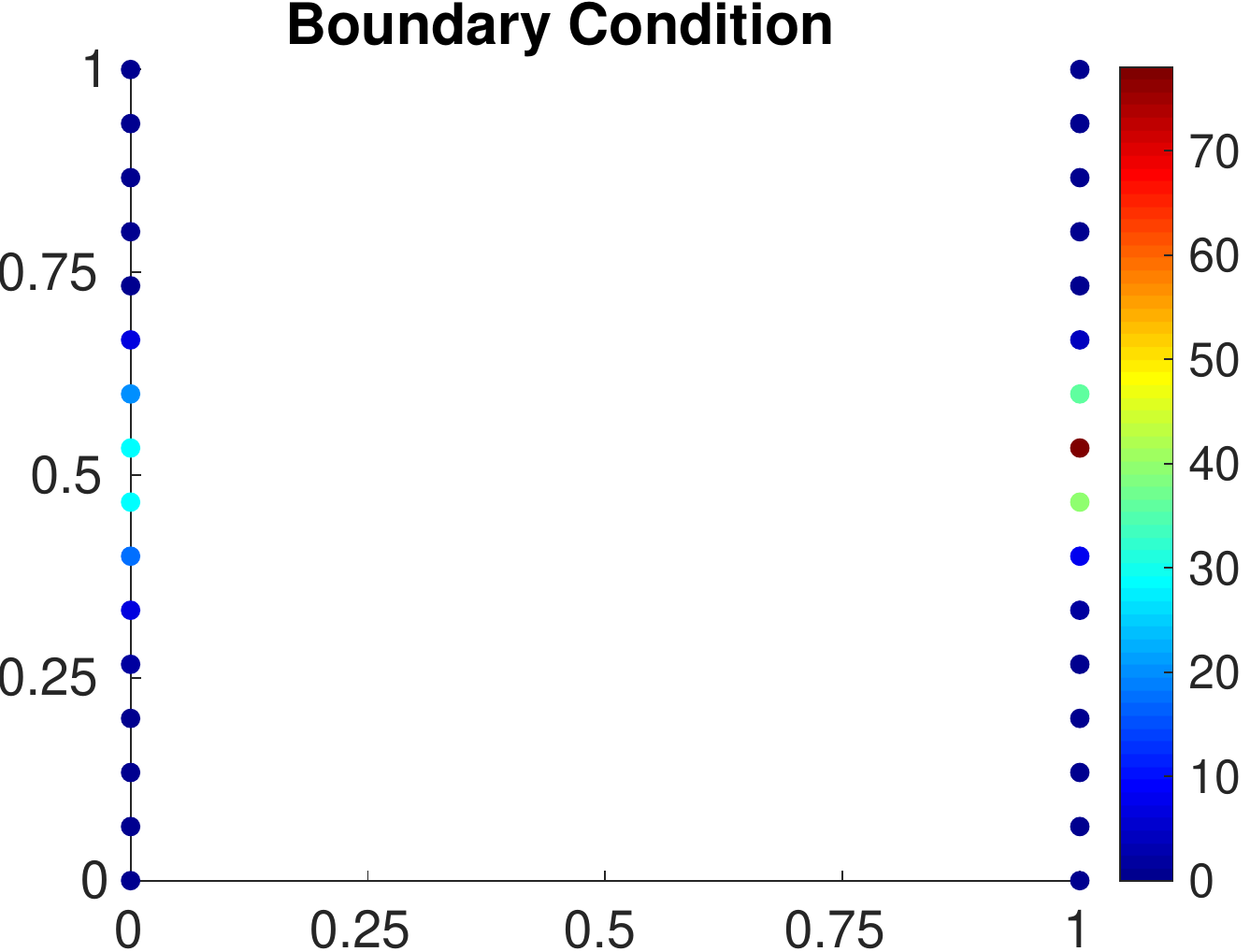}
    \includegraphics[width=0.33\textwidth]{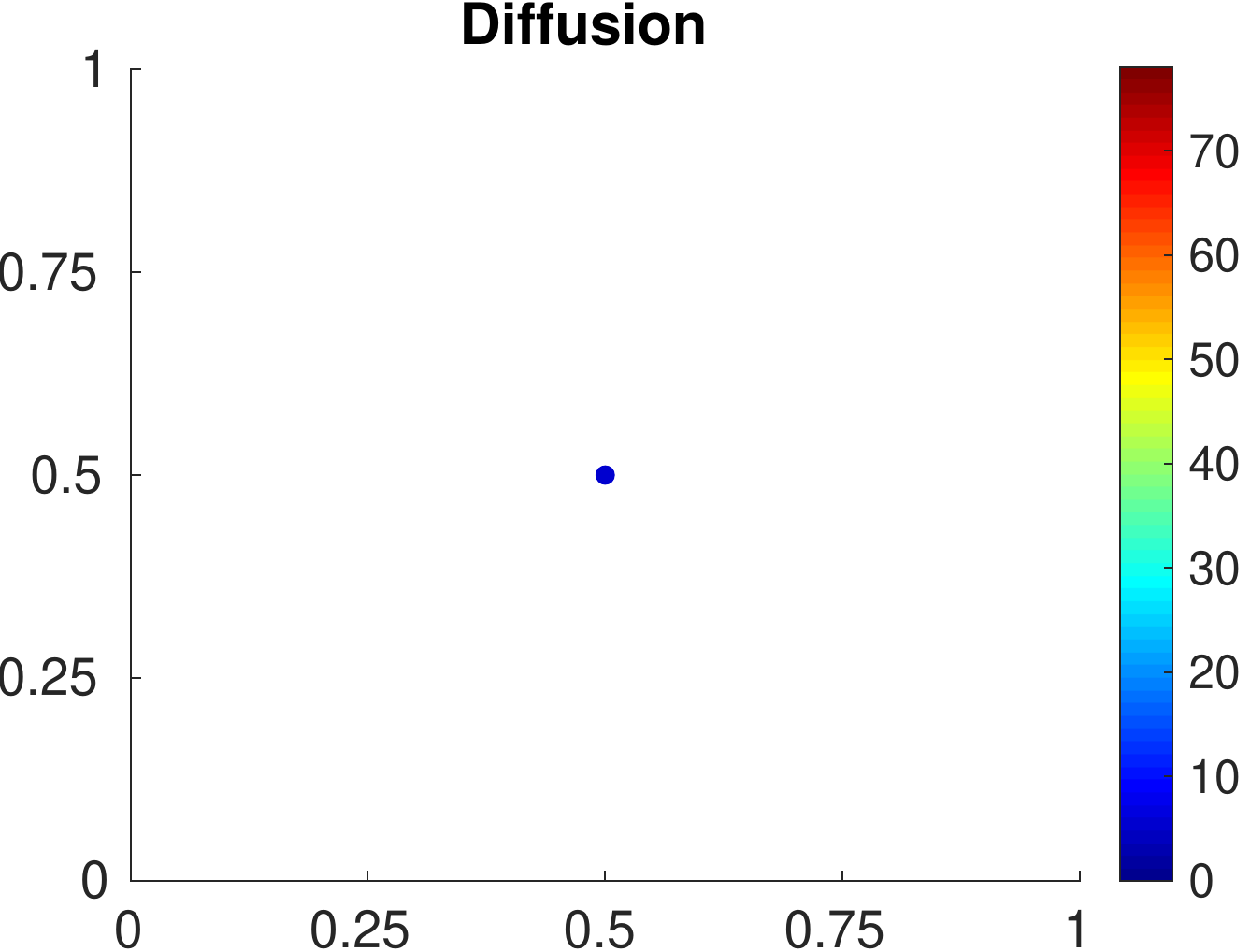}
    \caption{Sensitivity indices for the parameterization of the permeability field (left), pressure equation Dirichlet boundary conditions (center), and diffusion coefficient (right).}
    \label{fig:source_est_sensitivities}
\end{figure}

We see greater sensitivity in localized regions of the Dirichlet boundary conditions around $y=0.5$ (at both $x=0$ and $x=1$) and the permeability field in the region around $x=0$ and $y=0.5$. We
also observe a slight asymmetry which is consistent with the asymmetry
observed in the state solution. The diffusion coefficient sensitivity
is significantly smaller than the high sensitivity regions of the Dirichlet boundary conditions and permeability field.
The set sensitivity indices for the permeability field, left and right Dirichlet boundary conditions, and diffusion coefficient
are given in Table~\ref{tab:set_sensitivities}. 

\begin{table}[!ht]
    \centering
    \begin{tabular}{c|c|c|c|c}
    Set of Parameters& Permeability & Left BC & Right BC & Diffusion\\ \hline
    Set Sensitivity Index& $41.53$ & $40.25$ & $80.05$ & $4.92$ \\
     \end{tabular}
    \caption{Estimated set sensitivity indices for the parameters
      corresponding to the permeability field, left Dirichlet boundary
      condition, right Dirichlet boundary condition, and diffusion
      coefficient.}
    \label{tab:set_sensitivities}
\end{table}

The computation was performed using 32 processors and 32 random
vectors in the eigenvalue solver. This enables the eigenvalue solver
to execute all 32 matrix vector products simultaneously and reduce the
overall execution time to approximately 4 KKT solves.

\section{Conclusion}
\label{sec:conclusion}

We have introduced hyper-differential sensitivity analysis (HDSA)
arguing for a goal-oriented paradigm for sensitivity analysis in the
context of PDE-constrained optimization. We have shown through an
analytic example that HDSA is different from traditional sensitivity
analysis.  The complication of this approach is that directional
derivatives are required through the optimality conditions, which
includes partial differential equations as constraints.  The high
computational requirements are addressed through efficient
adjoint-based methods for the inner optimization problem and
randomized generalized eigenvalue solvers to estimate the sensitivity
indices. We explore three levels of parallelism including the
underlying linear algebra constructs, parallel randomized algorithms,
and global sensitivity sampling procedures.  Two numerical examples
demonstrate the flexibility of HDSA in that uncertainties in both
control and inverse problems can be addressed, in addition to
demonstrating the extensibility to different sets of PDEs (steady
state and transient).  HDSA provides a computational tractable
approach to prioritize large numbers of uncertain parameters relevant
to the optimal solution of PDE-constrained optimization problems.  A
global sensitivity strategy has been developed with intuition through
a theoretical bound that depends on the underlying nonlinearity.  The
numerical examples demonstrate the challenge managing uncertainties in
an optimization problem and how HDSA provides critical insight that
can guide subsequent analysis such as field measurements, laboratory
experiments, physics development, and robust optimization.

HDSA facilitates new exploration by asking and answering questions
that augment other UQ methodologies in the service of optimal design
and decision making. A particular focus is given to multi-physics
applications and the computational efficiency needed to explore high
dimensional parameter spaces which contain spatial and/or temporally
dependent parameters corresponding to a variety of physical
quantities. The abstraction and generality of the proposed method
permits various extensions and applications. For instance, we may
consider parameters in the objective function such as weights, data,
or user specified algorithmic parameters. The parallelism of our
method, matrix free software design, and underlying Trilinos
constructs facilitates computationally scalability. By combining
mathematical abstraction, efficient software infrastructure, and
exploitation of low rank structure (when present), HDSA enables us to
explore a variety of UQ questions in the context of optimization
constrained by PDEs.

\acknowledgements

The authors are grateful to Arvind Saibaba for helpful discussions
facilitating this work. This paper describes objective technical
results and analysis. Any subjective views or opinions that might be
expressed in the paper do not necessarily represent the views of the
U.S. Department of Energy or the United States Government. Sandia
National Laboratories is a multimission laboratory managed and
operated by National Technology and Engineering Solutions of Sandia
LLC, a wholly owned subsidiary of Honeywell International, Inc., for
the U.S. Department of Energy's National Nuclear Security
Administration under contract DE-NA-0003525. SAND2019-10626 J.

\bibliographystyle{IJ4UQ_Bibliography_Style}

\bibliography{dasco}

\begin{thebibliography}{10}

\bibitem{Vogel_99}
Vogel, C.R., Sparse matrix computations arising in distributed parameter
  identification, {\em SIAM J. Matrix Anal. Appl.}, pp. 1027--1037, 1999.

\bibitem{Archer_01}
Ascher, U.M. and Haber, E., Grid refinement and scaling for distributed
  parameter estimation problems, {\em Inverse Problems}, 17:571--590, 2001.

\bibitem{Haber_01}
Haber, E. and Ascher, U.M., Preconditioned all-at-once methods for large,
  sparse parameter estimation problems, {\em Inverse Problems}, 17:1847--1864,
  2001.

\bibitem{Vogel_02}
Vogel, C.R., {\em Computational Methods for Inverse Problems}, SIAM Frontiers
  in Applied Mathematics Series, 2002.

\bibitem{Biegler_03}
L.~T.~Biegler, O.~Ghattas, M.H. and van Bloemen~Waanders, B. (Eds.), {\em
  Large-Scale PDE-Constrained Optimization}, Vol.~30, Springer-Verlag Lecture
  Notes in Computational Science and Engineering, 2003.

\bibitem{Biros_05}
Biros, G. and Ghattas, O., Parallel {L}agrange-{N}ewton-{K}rylov-{S}chur
  methods for {PDE}-constrained optimization. {Parts I-II}, {\em SIAM J. Sci.
  Comput.}, 27:687-- 738, 2005.

\bibitem{Laird_05}
Laird, C.D., Biegler, L.T., van Bloemen~Waanders, B., and Bartlett, R.A., Time
  dependent contaminant source determination for municipal water networks using
  large scale optimization, {\em ASCE J. Water Res. Mgt. Plan.}, pp. 125--134,
  2005.

\bibitem{Hintermuller_05}
Hinterm{\"u}ller, M. and Vicente, L.N., Space mapping for optimal control of
  partial differential equations, {\em SIAM J. Opt.}, 15:1002--1025, 2005.

\bibitem{Hazra_06}
Hazra, S.B. and Schulz, V., Simultaneous pseudo-timestepping for aerodynamic
  shape optimization problems with state constraints, {\em SIAM J. Sci.
  Comput.}, 28:1078--1099, 2006.

\bibitem{Biegler_07}
Biegler, L.T., Ghattas, O., Heinkenschloss, M., Keyes, D., and van
  Bloemen~Waanders, B. (Eds.), {\em Real-Time PDE-Constrained Optimization},
  Vol.~3, SIAM Computational Science and Engineering, 2007.

\bibitem{Borzi_07}
Borzi, A., High-order discretization and multigrid solution of elliptic
  nonlinear constrained optimal control problems, {\em J. Comp. Applied Math},
  200:67--85, 2007.

\bibitem{Hinze_09}
Hinze, M., Pinnau, R., Ulbrich, M., and Ulbrich, S., {\em Optimization with PDE
  Constraints}, Springer, 2009.

\bibitem{Biegler_11}
Biegler, L., Biros, G., Ghattas, O., Heinkenschloss, M., Keyes, D., Mallick,
  B., Marzouk, Y., Tenorio, L., van Bloemen~Waanders, B., and Willcox, K.
  (Eds.), {\em Large-Scale Inverse Problems and Quantification of Uncertainty},
  John Wiley and Sons, 2011.

\bibitem{intro_SA_uq_handbook}
Iooss, B. and Saltelli, A.
\newblock Introduction to sensitivity analysis.
\newblock In: Ghanem, R., Higdon, D., and Owhadi, H. (Eds.), {\em Handbook for
  Uncertainty Quantification}, pp. 1103--1122. Springer, 2016.

\bibitem{saltellibook}
Saltelli, A., Ratto, M., Andres, T., Campolongo, F., Cariboni, J., Gatelli, D.,
  Saisana, M., and Tarantola, S., {\em Global sensitivity analysis: the
  primer}, Wiley, 2008.

\bibitem{iooss}
{Iooss}, B. and {Lema\^itre}, P.
\newblock A review on global analysis methods.
\newblock In: {Dellino}, G. and {Meloni}, C. (Eds.), {\em Uncertainty
  management in simulation-optimization of complex systems}, chapter~5, pp.
  543--501. Springer, 2015.

\bibitem{borgonovo2}
Borgonovo, E., A new uncertainty importance measure, {\em Reliability Eng. Sys.
  Safety}, 92:771--784, 2007.

\bibitem{kucherenko_derivative}
Kucherenko, S. and Iooss, B.
\newblock {\em Handbook of Uncertainty Quantification}, chapter
  Derivative-based global sensitivity measures.
\newblock Springer, 2016.

\bibitem{variance_based_uq_handbook}
Prieur, C. and Tarantola, S.
\newblock Variance-based sensitivity analysis: Theory and estimation
  algorithms.
\newblock In: Ghanem, R., Higdon, D., and Owhadi, H. (Eds.), {\em Handbook for
  Uncertainty Quantification}, pp. 1217--1239. Springer, 2016.

\bibitem{staum}
Song, E., Nelson, B.L., and Staum, J., Shapley effects for global sensitivity
  analysis: Theory and computation, {\em SIAM/ASA J. Uncertain. Quantif.},
  4:1060--1083, 2016.

\bibitem{griesse2}
Brandes, K. and Griesse, R., Quantitative stability analysis of optimal
  solutions in {PDE}-constrained optimization, {\em Journal of Computational
  and Applied Mathematics}, 206:908--926, 2007.

\bibitem{griesse_constraints}
B{\"u}skens, C. and Griesse, R., Parametric sensitivity analysis of perturbed
  {PDE} optimal control problems with state and control constraints, {\em
  Journal of Optimization Theory and Applications}, 131(1):17--35, 2006.

\bibitem{Griesse_part_2}
Griesse, R., Parametric sensitivity analysis in optimal control of a
  reaction-diffusion system -- part {II}: practical methods and examples, {\em
  Optimization Methods and Software}, 19(2):217--242, 2004.

\bibitem{Griesse_part_1}
Griesse, R., Parametric sensitivity analysis in optimal control of a reaction
  diffusion system. {I}. solution differentiability, {\em Numerical Functional
  Analysis and Optimization}, 25(1-2):93--117, 2004.

\bibitem{Griesse_Thesis}
Griesse, R.
\newblock Stability and sensitivity analysis in optimal control of partial
  differential equations.
\newblock Habilitation Thesis, Faculty of Natural Sciences, Karl-Franzens
  University, 2007.

\bibitem{Griesse_SISC}
Griesse, R. and Vexler, B., Numerical sensitivity analysis for the quantity of
  interest in {PDE-Constrained} optimization, {\em SIAM Journal on Scientific
  Computing}, 29(1):22--48, 2007.

\bibitem{griesse_3d}
Griesse, R. and Volkwein, S., Parametric sensitivity analysis for optimal
  boundary control of a 3{D} reaction-difusion system, In: Pillo, G.D. and
  Roma, M. (Eds.), {\em Nonconvex Optimization and its Applications}, Vol.~83.
  Springer, Berlin, 2006.

\bibitem{Griesse_AD}
Griesse, R. and Walther, A., Parametric sensitivities for optimal control
  problems using automatic differentiation, {\em Optimal Control Applications
  and Methods}, 24:297--314, 2003.

\bibitem{arvind}
Saibaba, A.K., Lee, J., and Kitanidis, P.K., Randomized algorithms for
  generalized {H}ermitian eigenvalue problems with application to computing
  {K}arhunen-{L}o{\`e}ve expansion, {\em Numerical Linear Algebra with
  Applications}, 23:314--339, 2016.

\bibitem{Murthy_OR}
Murthy, P.R., {\em Operations Research}, New Age International Publishers, 2nd
  edition, 2007.

\bibitem{shapiro_SIAM_review}
Bonnans, J.F. and Shapiro, A., Optimization problems with perturbations: A
  guided tour, {\em SIAM Review}, 40(2):228--264, 1998.

\bibitem{delsa}
Rakovec, O., Hill, M.C., Clark, M.P., Weerts, A.H., Teuling, A.J., and
  Uijlenhoet, R., Distributed evaluation of local sensitivity analysis (delsa),
  with application to hydrologic models, {\em Water Resources Research},
  50:409--426, 2014.

\bibitem{dgsm1}
{Sobol'}, I. and {Kucherenko}, S., Derivative based global sensitivity measures
  and the link with global sensitivity indices, {\em Math. Comp. Simul.},
  79:3009--30017, 2009.

\bibitem{dgsm2}
{Sobol'}, I. and {Kucherenko}, S., A new derivative based importance criterion
  for groups of variables and its link with the global sensitivity indices,
  {\em Comput. Phys. Comm.}, 181:1212--1217, 2010.

\bibitem{morris}
{Morris}, M., Factorial sampling plans for preliminary computational
  experiments, {\em Technometrics}, 33:161--174, 1991.

\bibitem{active_subspaces}
Constantine, P.G., {\em Active Subspaces: Emerging Ideas for Dimension
  Reduction in Parameter Studies}, {SIAM}, 2015.

\bibitem{activity_scores}
Constantine, P.G. and Diaz, P., Global sensitivity metrics from active
  subspaces, {\em Reliability Engineering \& System Safety}, 162:1--13, 2017.

\bibitem{structured_eigen}
Fassbender, H. and Kressner, D., Structured eigenvalue problems, {\em
  GAMM-Mitteilungen}, 29(2):297--318, 2006.

\bibitem{randomized_la_review}
Halko, N., Martinsson, P.G., and Tropp, J.A., Finding structure with
  randomness: Probabilistic algorithms for constructing approximate matrix
  decompositions, {\em SIAM Review}, 53(2):217--288, 2011.

\bibitem{cvd}
Ito, K. and Ravindran, S.S., Optimal control of thermally convected fluid
  flows, {\em SIAM J. on Scientific Computing}, 19:1847--1869, 1998.

\bibitem{Kouri2018}
Kouri, D.P. and Ridzal, D.
\newblock {\em Inexact Trust-Region Methods for PDE-Constrained Optimization},
  pp. 83--121.
\newblock Springer New York, New York, NY, 2018.

\bibitem{rol}
Kouri, D.P., von Winckel, G., and Ridzal, D.
\newblock {ROL}: {R}apid {O}ptimization {L}ibrary.

\bibitem{Trilinos-Overview}
Heroux, M., Bartlett, R., Howle, V., Hoekstra, R., Hu, J., Kolda, T., Lehoucq,
  R., Long, K., Pawlowski, R., Phipps, E., Salinger, A., Thornquist, H.,
  Tuminaro, R., Willenbring, J., and Williams, A., An overview of {T}rilinos,
  Tech.~Rep. SAND2003-2927, Sandia National Laboratories, 2003.

\end{thebibliography}
\end{document}